\documentclass[12pt]{article}
\usepackage[american]{babel}
\usepackage[utf8]{inputenc}
\usepackage{cite}
\usepackage{geometry}
\usepackage{amsmath, amsthm, latexsym, amssymb}
\usepackage{amsfonts}
\usepackage{booktabs}  
\usepackage{graphicx} 
\usepackage{listings}
\usepackage{hyperref}
\usepackage{anysize}
\usepackage{caption}
\usepackage{subcaption}
\usepackage{placeins}
\usepackage{mathtools}
\usepackage{todonotes}
\usepackage{tcolorbox}
\usepackage{bm}
\usepackage{rotating,tabularx,siunitx}
\usepackage[title]{appendix}
\usepackage[ruled, lined, linesnumbered]{algorithm2e}

\usepackage{environ}

\newtheorem{theorem}{Theorem}[section]

\newtheorem{lemma}[theorem]{Lemma}

\newcounter{parentAlgoLine}
\SetKwBlock{BEGIN}{}{}

\makeatletter
\NewEnviron{substeps}{%
  \refstepcounter{AlgoLine}%
  \protected@edef\theparentequation{\theequation}%
  \setcounter{parentAlgoLine}{\value{AlgoLine}}%
  \setcounter{AlgoLine}{0}%
  \BEGIN{\BODY}%
  \setcounter{AlgoLine}{\value{parentAlgoLine}}%
  \ignorespacesafterend
}
\makeatother

\title
{Robust Implicit  Adaptive   Low Rank  Time-Stepping Methods for  Matrix Differential Equations }

\author{ 
Daniel Appel\"{o}
\thanks{Department of Mathematics, Virginia Tech,
Blacksburg, VA 24061 U.S.A.
 {\tt appelo@vt.edu} Research supported by DOE Office of Advanced Scientific Computing Research under the Advanced Research in Quantum Computing program, subcontracted from award 2019-LLNL-SCW-1683, NSF DMS-2208164, and Virginia Tech.
 }
\and
 Yingda Cheng
\thanks{Department of Mathematics, Virginia Tech,
Blacksburg, VA 24061 U.S.A.
 {\tt yingda@vt.edu}  Research supported by NSF DMS-2011838, DOE grant DE-SC0023164 and Virginia Tech.}}

\date{\today}

\begin{document}

\maketitle

\begin{abstract}
In this work, we develop implicit  rank-adaptive schemes for time-dependent matrix differential equations. The   dynamic low rank approximation (DLRA)  is a well-known technique to capture the dynamic low rank structure  based on Dirac–Frenkel time-dependent variational principle. In recent years, it has attracted a lot of attention due to its wide applicability.  Our schemes are inspired by the  three-step procedure used in the rank adaptive version of the unconventional robust integrator (the so called BUG integrator) \cite{ceruti2022rank} for DLRA. First, a prediction (basis update) step is made computing the approximate column and row spaces at the next time level. Second, a Galerkin evolution step is invoked using an implicit {solves} for the small core matrix. Finally, a truncation is made according to a prescribed error threshold. Since the DLRA is evolving the differential equation projected on to the tangent space of the low rank manifold, the error estimate of the BUG integrator contains the tangent projection (modeling) error which cannot be easily controlled by mesh refinement.
This can cause convergence issue for  equations with cross terms.

To address this issue, we propose a simple modification, consisting of merging the row and column spaces from the explicit step truncation method together with the BUG spaces in the prediction step. In addition, we propose an  adaptive strategy   where the BUG spaces are only computed if the residual for the solution obtained from the prediction space by explicit step truncation method, is too large. We prove stability and estimate the local truncation error of the schemes under assumptions. We benchmark the schemes in several tests, such as anisotropic diffusion, solid body rotation and the combination of the two, to show robust convergence properties.
\end{abstract}
 
\section{Introduction}
In this work, we are interested in solving linear matrix differential equations in the following form
\begin{equation}
\label{eqn:mode}
\frac{d}{dt}X(t)=F(X(t),t),  \quad X(t) \in \mathbb{R}^{m_1\times m_2}, \quad X(0)=X_0,
\end{equation}
where 
\[
F(X(t),t)=\sum_{j=1}^{s_{\rm sep}} A_j X(t) B_j^T +G(t),
\] 
and $A_j \in \mathbb{R}^{m_1\times m_1}, B_j \in \mathbb{R}^{m_2\times m_2}$ are sparse or structured matrices for which a fast matrix-vector product is assumed to be known. Further, $G(t) \in  \mathbb{R}^{m_1\times m_2} $ is a given function that is assumed to have a known low rank decomposition. The integer ${s_{\rm sep}}$, which denotes the separation rank of $F(\cdot),$  is assumed not to be too large. We focus on implicit methods, which are particularly needed for computing solutions to stiff equations.  

Equation \eqref{eqn:mode} arises in many applications governed by partial differential equations (PDE). For example, using the method of lines approach, we can cast numerical discretizations  of  two-dimensional (or even higher dimensional) linear time-dependent convection-diffusion-reaction equations in the form of \eqref{eqn:mode}. If we represent the unknowns on a  Cartesian grid, the elements of $X(t)$ will represent the unknowns (e.g. point values of the solution) at a given time $t.$ The matrices $A_j, B_j$ will, e.g., correspond to difference operators or variable coefficients, and $G(t)$ encodes the given source terms and boundary conditions. Compared to the traditional approach of numerical PDE solvers, the matrix approach does not vectorize the unknowns into a single vector at each time step, but rather we represent the unknowns in a matrix format, retaining the relative relations in each spatial dimension. 

In recent years, the matrix (or tensor in higher dimensions) approach received increased attention in numerical analysis because it offers a way to tame \emph{the curse of dimensionality}, enabling approximations to solutions of high dimensional PDEs \cite{hackbusch2012tensor,khoromskij2018tensor}. The overarching idea is to expose low rank structure in the solution manifold in order to drastically reduce storage and computation. This is the idea of  low rank tensor methods for PDEs \cite{grasedyck2013literature,bachmayr2023low}. In two dimensions, this approach reduces to the low rank matrix method, where it is assumed that the matrix $X(t)$ has a small rank $r.$  As is well known, a rank $r$ matrix of size $m_1 \times m_2$  allows for a SVD decomposition with only $(m_1+m_2-r)r$ degrees of freedom. This is much smaller than $m_1m_2$,  the degrees of freedom of the full matrix, when $r$ is small compared to $m_1$ and $m_2$.

If the rank of $X(t)$ is fixed, i.e. we constrain the solution $X(t)$ to live on the rank $r$ matrix manifold
$\mathcal{M}_r=\{X \in  \mathbb{R}^{m_1\times m_2}, {\sf rank}(X)=r\},$
then an effective numerical solution can be computed by the dynamic low rank approximation (DLRA) \cite{koch2007dynamical,lubich2014projector}. The DLRA solves
$
\frac{d}{dt}X(t)=\Pi_{X(t)} F(X(t),t),  
$
where $\Pi_{X(t)}$ is the orthogonal projection onto the tangent space $T_{X(t)}\mathcal{M}_r$ of the manifold $\mathcal{M}_r$ at $X(t).$ The DLRA has found success in many applications, including extensions to various tensor formats  \cite{lubich2018time, lubich2015time}. However, the constraint  that $X(t) \in \mathcal{M}_r$ requires that the rank $r$ is estimated a-priori and this is in contrast with the rank adaptive schemes.

A rank adaptive method adaptively choose the rank at each time step according to a prescribed error tolerance. Several approaches to rank adaptivity can be found in the literature. For example, a straightforward  idea is the so-called step truncation method \cite{dektor2021rank} which evolves the low rank solution for one time step by a traditional time stepping method in an ambient space of higher rank, then performs a truncation (by SVD with given tolerance). This method is intuitive and effective for explicit time stepping methods. For implicit schemes  \cite{rodgers2023implicit}, implementation of step truncation methods is nontrivial and requires an effective iterative solver in adaptive low rank format  \cite{grasedyck2013literature,bachmayr2023low}. Another approach that has been suggested by multiple groups, and that is similar to what we propose, is to modify the DLRA method by using rank increase and decrease indicators and augment or truncate the spaces at each time step \cite{hochbruck2023rank,hauck2023predictor,PhysRevB.102.094315}. In particular the ideas in \cite{PhysRevB.102.094315}, while being specialized to time dependent Schr\"{o}dinger, are similar to the approach we take here. A promising idea in \cite{ceruti2022rank} uses the so called unconventional (BUG) integrator \cite{ceruti2022unconventional} to achieve rank adaptivity. This approach was recently extended to higher order in time in \cite{Nakao:2023aa}. Finally we note the recent paper, \cite{kahza2024krylovbasedadaptiverankimplicittime}, that uses the extended Krylov subspace method,  \cite{simoncini2007new} to construct adaptive-rank implicit time integrators. 

In the BUG integrator \cite{ceruti2022unconventional}, two subproblems that come from the projector splitting of DLRA \cite{lubich2014projector} are solved. Then, the solutions to the two subproblems are used to update the column and row spaces of the matrix; and a Galerkin evolution step is performed in the resulting space, searching for the coefficients of a small matrix. For the rank adaptive version of the unconventional integrator \cite{ceruti2022rank}, this translates to solving two matrix differential equation of sizes $m_1\times r, m_2 \times r$ in the first step and one matrix differential equation of size $2r\times 2r$ in the Galerkin step, followed by a truncated SVD step. The ideas in \cite{ceruti2022unconventional,ceruti2022rank} can be implemented for implicit methods because the matrix size is fixed in each subproblems, so a standard linear solver will suffice. However, the errors of both methods in \cite{ceruti2022unconventional,ceruti2022rank}  are subject to the tangent projection error or the so called \emph{modeling error} \cite{kieri2019projection}.

It is important to understand that this modeling error is present in many problems and can, as we will see below, result in non convergent numerical methods. We now describe a trivial but prototypical example that will make the DLRA/BUG method fail. Assume that $X(t) = U \Sigma V^T$ and that the column and row space of the matrix $F(X(t),t)$ are both (respectively) orthogonal to $U$ and $V$, then $\Pi_{X(t)} F(X(t),t)=0,$ which means the solution to the DLRA description will remain stationary. For example, this will happen for the numerical discretization when $X(t)$ is a  rank 1 even function in both variables (say $X(t,x_1,x_2) = \exp(-x_1^2) \exp(-x_2^2)$) on a square domain with center $(0,0)$, and  $F(X(t),t) = AXB^T$, with $A = {\sf diag}(x_1)$ and $B={\sf diag}(x_2)$. 

The main contribution of this paper is to propose a simple improvement to the BUG integrator that enhances the robustness with respect to convergence. In our method, we merge the row and column spaces from the explicit step truncation method with the BUG spaces in the prediction step to alleviate the tangent projection error associated with DLRA. As the BUG spaces can be  expensive to compute when implicit schemes are used, we propose an adaptive acceleration strategy based on the residual of the underlying classic scheme. The strategy is straightforward, if the residual obtained when {\bf only} the step truncation spaces are used, is larger than the truncation threshold then we add in the BUG spaces. We observe that this strategy is robust in convergence and can reduce the computational time for moderately stiff problems.  

The rest of the paper is organized as follows. Section \ref{sec:background} reviews the background on low rank time integrators.
In Section \ref{sec:implicit}, we describe the proposed schemes and perform numerical analysis of their properties.
 Section \ref{sec:numerical} provides numerical experiments and in Section \ref{sec:conclusion} we draw conclusions.

\section{Background review}
\label{sec:background}
In this section, we will  review basic concepts of low rank time stepping methods for \eqref{eqn:mode}. In Section \ref{sec:preliminary}, we gather the notations used in the paper. Section \ref{sec:explicit} reviews an explicit scheme truncation method based on the forward Euler method, while Section \ref{sec:dlra} reviews the rank adaptive BUG scheme.
\subsection{Notations and preliminaries}
\label{sec:preliminary}
In this paper, we use $\|\cdot\|$ to denote the matrix Frobenius norm. The Frobenius norm is a natural choice for matrix functions as $\|A\|$ coincides with the $L^2$ vector norm of the vectorized matrix ${\sf vec}(A),$ which is commonly used in the analysis of the standard ODE/PDE solvers. We also use matrix inner product defined as follows, for $A, B  \in \mathbb{R}^{m_1\times m_2},$ we define $\langle A, B \rangle =\sum_{i,j=1}^{m_1, m_2} A_{ij} B_{i,j}.$ Then $\langle A, A \rangle =\|A\|^2.$ The following property, 
\begin{equation}
\label{eqn:innerp}
\langle A, U S V^T \rangle = \langle U^T A V, S   \rangle, 
\end{equation}
holds for any $U \in \mathbb{R}^{m_1\times r}, S \in \mathbb{R}^{r\times s}, V \in \mathbb{R}^{m_2\times s}$. Here we stop to caution the reader that throughout this paper a matrix $\Sigma, S,$ or $\mathcal{S}$ does {\it not} always represent the diagonal matrix holding the singular values of another matrix.

 The numerical solution at $t^n$ is denoted by $\hat{X}^n \in \mathbb{R}^{m_1 \times m_2}.$ In particular, this is a rank $r_n$ matrix with singular value decomposition (SVD) $\hat{X}^n=U^n \Sigma^n (V^n)^T,$ where $U^n \in \mathbb{R}^{m_1\times r_n}$, $V^n \in \mathbb{R}^{m_2\times r_n}$ have orthogonal columns, and $\Sigma^n={\sf diag}(\sigma_1, \cdots, \sigma_{r_n})$ is a $r_n\times r_n$ diagonal matrix with diagonal entries $\sigma_1\ge \cdots \ge \sigma_{r_n} >0.$ Here, the rank $r_n$ will be chosen adaptively by the numerical scheme so that $\hat{X}^n$ approximate the true solution $X(t^n)$ with a prescribed accuracy.  
 
 We use $\mathcal{T}_\epsilon$ to denote a generic matrix approximation operator with accuracy $\epsilon$ in the Frobenius norm, i.e. $\|A -\mathcal{T}_\epsilon(A)\| \le \epsilon.$  A prominent example is the  truncated SVD of a matrix.  Namely, given a generic rank $r$ matrix $A \in \mathbb{R}^{m_1\times m_2}$ with reduced SVD: $A=U\Sigma V^T,$ then $\mathcal{T}^{svd}_\epsilon(A) = U[:,1:s] \,{\sf diag}(\sigma_1, \cdots, \sigma_{s}) \,V[:,1:s]^T,$ where we have used the standard MATLAB notation for  submatrices and ${r}$ is chosen to be the smallest integer so that $(\sum_{j=s+1}^r \sigma_j^2)^{1/2} \le \epsilon.$ 
 
In this paper, we will also encounter low rank matrix sum operations. This frequently appearing operation in low rank methods can   be efficiently computed by Algorithm \ref{algo:msum}. The inputs to this algorithm are the SVD representations of the matrices $X_j=U_j \Sigma_j (V_j)^T, j=1\ldots m,$ and the output is the truncated SVD of their sum $\mathcal{T}^{sum}_\epsilon(\sum_{j=1}^m X_j)=\mathcal{U}\mathcal{S}\mathcal{V}^T.$
We first write $U=[U_1, \ldots, U_m],$ $\Sigma={\sf diag}(\Sigma_1, \ldots, \Sigma_m),$ $V=[V_1, \ldots, V_m]$. Then use the  column-pivoted QR decomposition, \cite{golub2013matrix}, denoted by \mbox{$[Q, R, \Pi]={\sf qr}(A)$}, which computes $QR\Pi=A.$ In this factorization the columns of $Q$ are orthogonal as usual, but the introduction of the permutation matrix $\Pi$, makes it possible to guarantee that the diagonal entries of the upper triangular matrix $R$ are strictly decreasing in magnitude. Assume that the column pivoted QR procedure applied to to $U$ and $V$ yields the factorizations $Q_1R_1\Pi_1$, and $Q_2R_2\Pi_2$. Then the truncated matrix sum 
$
\mathcal{T}^{sum}_\epsilon(\sum_{j=1}^m X_j)=Q_1 \mathcal{U}\mathcal{S}(Q_2 \mathcal{V})^T.
$
is obtained by a truncated SVD on the small matrix $\mathcal{T}^{svd}_{\epsilon}(R_1 \Pi_1 \, \Sigma \, \Pi_2^T R_2^T)=\mathcal{U}\mathcal{S}\mathcal{V}^T$. 

 \begin{algorithm}[htbp]
    \SetKwFunction{isOddNumber}{isOddNumber}
    \SetKwInOut{KwIn}{Input}
    \SetKwInOut{KwOut}{Output}
      \SetKwInOut{KwPara}{Parameter}

    \KwIn{low rank matrices $X_j, j=1\ldots m$ or their SVD $U_j \Sigma_j (V_j)^T, j=1\ldots m$}
    \KwOut{truncated SVD of their sum $\mathcal{T}^{sum}_\epsilon(\sum_{j=1}^m X_j)=\mathcal{U}\mathcal{S}\mathcal{V}^T$}
	\KwPara{tolerance $\epsilon$}
 	Form $U=[U_1, \ldots, U_m], \, \Sigma={\sf diag}(\Sigma_1, \ldots, \Sigma_m), \,V=[V_1, \ldots, V_m]$.

	Perform column pivoted QR: $ [Q_1, R_1, \Pi_1]={\sf qr}(U),$ $[Q_2, R_2, \Pi_2]={\sf qr}(V).$
	
	Compute the truncated SVD:  $\mathcal{T}^{svd}_{\epsilon}(R_1 \Pi_1 \,\Sigma \, \Pi_2^T R_2^T)=\mathcal{U} \mathcal{S} \mathcal{V}^T$.
	
	Form $\mathcal{U} \leftarrow Q_1 \mathcal{U}, \mathcal{V} \leftarrow Q_2 \mathcal{V}$.
	
	Return $[\mathcal{U}, \mathcal{S}, \mathcal{V}]=\mathcal{T}^{sum}_\epsilon(\sum_{j=1}^m X_j)$
    \caption{Sum of low rank matrices}
    \label{algo:msum}
\end{algorithm}

\subsection{Explicit step truncation schemes}
\label{sec:explicit}
 The primary assumption to guarantee the efficiency of any low rank solver is that $r_n \ll m_1,  r_n \ll m_2,$ i.e. the solution can be well approximated by a low rank matrix. 
 When such assumptions hold, the main steps of the numerical algorithm should specify the evolution of $U^n, \Sigma^n, V^n$ based on the matrix differential equation \eqref{eqn:mode}. For explicit schemes, a simple approach is to  numerically integrate $\hat{X}^n$ according to a standard time integrator to $t^{n+1}$ and then perform a truncated SVD according to the error threshold to obtain $\hat{X}^{n+1}.$ This is the so-called step truncation method. Methods of this type and their variations have been discussed in \cite{kieri2019projection,dektor2021rank}.  
The main steps of the rank adaptive forward Euler time scheme are highlighted in Algorithm \ref{algo:fe}.
The first step of Algorithm \ref{algo:fe} is a standard forward Euler step for \eqref{eqn:mode} with the truncated right hand side in Line \ref{a1line1}. This will result in a numerical solution $\hat{X}^{n+1,pre}$ with higher rank than needed.  
 Then, the second step in Line \ref{a1line2} is to truncate the solution to a lower rank matrix.   In both steps, we use Algorithm \ref{algo:msum} to perform the sum of the low rank matrices involved.

\begin{algorithm}[htbp]
    \SetKwFunction{isOddNumber}{isOddNumber}
    \SetKwInOut{KwIn}{Input}
    \SetKwInOut{KwOut}{Output}
      \SetKwInOut{KwPara}{Parameter}

    \KwIn{numerical solution at $t^n:$  rank $r_n$ matrix $\hat{X}^n$ in its SVD form $U^n \Sigma^n (V^n)^T.$}
    \KwOut{numerical solution at $t^{n+1}:$  rank $r_{n+1}$ matrix $\hat{X}^{n+1}$ in its SVD form $U^{n+1} \Sigma^{n+1} (V^{n+1})^T.$}
	\KwPara{time step $\Delta t$, error tolerance $\epsilon_1, \epsilon_2$}
{\bf (Evolution).}  $\hat{X}^{n+1,pre}=\hat{X}^{n}+\Delta t \mathcal{T}^{sum}_{\epsilon_1}(F(\hat{X}^n,t^n)).$
\label{a1line1}

	 {\bf (Truncation).}  $\hat{X}^{n+1}=\mathcal{T}^{sum}_{\epsilon_2} (\hat{X}^{n+1,pre}).$
	 \label{a1line2}
    \caption{Forward Euler scheme   $t^n \rightarrow t^{n+1}$}
    \label{algo:fe}
\end{algorithm}

The convergence of this algorithm is well understood \cite{rodgers2023implicit}. The local truncation error is on the order of $O(\Delta t^2+\epsilon_1 \Delta t +\epsilon_2).$ If one choose $\epsilon_1=O(\Delta t), \epsilon_2=O(\Delta t^2),$ we will obtain a first order accurate solution. 
As for computational cost, we can see that the procedure involves QR factorization of tall matrices and SVD for a small matrix. If we assume $r_n =O(r),$ $s$ is small (i.e. $s=O(1)$) and $G$ has low rank (i.e. ${\sf rank}(G)=O(r)$),    the computational complexity is on the order of $O(m_1 r^2 +m_2 r^2 +r^3).$  Algorithm \ref{algo:fe}  can be readily extended to higher order by embedding in Runge-Kutta, multistep methods,  and was also considered  with tangent projection in the projected Runge-Kutta schemes \cite{kieri2019projection,dektor2021rank,guo2022low}.

\subsection{Dynamic low rank approximation}
\label{sec:dlra}

The DLRA modifies the equation, and  solves
\begin{equation}
\label{eq:DLRAF}
\frac{d}{dt}X(t)=\Pi_{X(t)} F(X(t), t),  
\end{equation}
where $\Pi_{X(t)}$ is the orthogonal projection onto the tangent space $T_{X(t)}\mathcal{M}_r$ of the the rank $r$ matrix manifold
$\mathcal{M}_r=\{X \in  \mathbb{R}^{m_1\times m_2}, {\sf rank}(X)=r\}$ at $X(t).$ The DLRA is particularly suited for a fixed rank calculation, if the computation is constrained on $\mathcal{M}_r.$ It has found success in many applications, including extensions to various tensor formats  \cite{lubich2018time, lubich2015time} and particularly in quantum mechanics \cite{paeckel2019time}.

We reminder the reader that if the SVD of $X \in \mathcal{M}_r$ is given by $U \Sigma V^T,$ then the tangent space of $\mathcal{M}_r$ at $X$ is given by
\[
T_{X}\mathcal{M}_r=\left \{ [ U \quad U_\perp] \begin{bmatrix}
\mathbb{R}^{r \times r} & \mathbb{R}^{r \times (m_2-r)} \\
\mathbb{R}^{(m_1-r)\times r} & 0^{(m_1-r) \times (m_2-r)} 
\end{bmatrix} [V \quad V_\perp]^T\right \},
\]
where $U_\perp, V_\perp$ are orthogonal complements of $U, V$ in $\mathbb{R}^{m_1}, \mathbb{R}^{m_2},$ respectively.
As we can see later in the PDE examples, typically if there are cross terms  like $F(X)=AXB$ (e.g. cross derivatives or variable coefficient problem with operators acting on both left and right side of the matrix), then in general $AXB \notin T_{X}\mathcal{M}_r.$ In this case, the \emph{modeling error} \cite{kieri2019projection} (i.e. the error from the tangent projection, which is $\|(I-\Pi_{X(t)})F(X,t)\|$) will be evident in the numerical approximation, see for example error estimates in \cite{ceruti2022unconventional,ceruti2022rank}. 

Below we will review the rank adaptive BUG integrator  in \cite{ceruti2022rank} for the discretization of \eqref{eq:DLRAF}. 
  First, two subproblems that come from the projector splitting of DLRA \cite{lubich2014projector} are solved in the $K-$ and $L-$step. Then, the solutions to the two subproblems are used to update the column and row spaces of the matrix; and a Galerkin evolution step is performed in the resulting space, searching for the optimal solution with the Galerkin condition. The idea of finding the appropriate subspace and computing the solution in that space is similar to a widely used technique in numerical linear algebra called projection methods \cite{saad2003iterative}. %
  
  For completeness, we describe the rank adaptive BUG integrator in Algorithm \ref{algo:dlra}. The original algorithm is written in a time continuous format, but here we present it using an implicit Euler discretization of the $K-, L-,S-$steps to facilitate the discussion in the remaining part of the paper.

\begin{algorithm}[htbp]
    \SetKwFunction{isOddNumber}{isOddNumber}
    \SetKwInOut{KwIn}{Input}
    \SetKwInOut{KwOut}{Output}
      \SetKwInOut{KwPara}{Parameter}
    \KwIn{numerical solution at $t^n:$  rank $r_n$ matrix $\hat{X}^n$ in its SVD form $U^n \Sigma^n (V^n)^T.$}
    \KwOut{numerical solution at $t^{n+1}:$  rank $r_{n+1}$ matrix $\hat{X}^{n+1}$ in its SVD form $U^{n+1} \Sigma^{n+1} (V^{n+1})^T.$}
	\KwPara{time step $\Delta t$, error tolerance $\epsilon_2$}  
	{\bf (Prediction).}   $K-$step and $L-$step integrating from $t^n$ to $t^{n+1}$.
	Solve $$ K^{n+1} - K^{n}=\Delta t F(K^{n+1} (V^n)^T,t) V^n, \quad K^n=U^n \Sigma^n,$$
	to obtain $K^{n+1},$ and $[\tilde{U},\sim,\sim]= {\sf qr}([U^n,  K^{n+1}]).$
	Solve $$ L^{n+1}-L^n=\Delta t F(U^n (L^{n+1})^T,t)^T U^n, \quad L^n=V^n \Sigma^n,$$
	to obtain $L^{n+1},$ and  $[\tilde{V},\sim,\sim]= {\sf qr}([V^n, L^{n+1}]).$  
	
{\bf (Galerkin Evolution).}  $S-$step: solve for $S^{n+1} $ from  
$$ S^{n+1}-S^{n}= \Delta t \tilde{U}^T F(\tilde{U} S^{n+1} \tilde{V}^T, t) \tilde{V}, \quad S^n= \tilde{U}^T U^n \Sigma^n (V^n)^T \tilde{V},$$
to obtain $S^{n+1}.$

	 {\bf (Truncation).}  $\hat{X}^{n+1}=\tilde{U}\mathcal{T}_{\epsilon_2}^{svd}(S^{n+1})\tilde{V}^T = U^{n+1} \Sigma^{n+1} (V^{n+1})^T.$
 
    \caption{Rank adaptive BUG integrator using implicit Euler \cite{ceruti2022rank}.}
    \label{algo:dlra}
\end{algorithm}

\section{Numerical method}
\label{sec:implicit}
We now present the proposed numerical methods, the Merge method and its adapted version. We first describe the algorithms, and then discuss the properties and the rationale behind the design of the schemes. %

\subsection{The Merge method}
In the Merge method, we merge the column and row spaces from the explicit step truncation method and the spaces from the $K-$ and $L-$ steps in the BUG method. As mentioned above the method consists of three stages. 
  
{\bf Prediction:} Our first ingredient in the method is to predict the column and row space using the explicit Euler scheme in combination with the BUG prediction spaces from Algorithm \ref{algo:dlra}.  Precisely, the predicted column and row spaces are defined as the column and row spaces of the collection of matrices
\[
[\hat{X}^n, \, \mathcal{T}_{\epsilon_1}(F(\hat{X}^n,t^n)), \, K^{n+1} (L^{n+1})^T].
\]
Here the truncation level $\epsilon_1$ can be set to zero (no truncation used) or chosen as $\epsilon_1 = C_1 \Delta t$. Both choices gives first order accurate solutions but depending on the form of (\ref{eqn:mode}) the added cost from the  $\epsilon_1 = C_1 \Delta t$ truncation may be offset by the reduction in the dimension of the predicted  row and column spaces. %
To complete the next step we first orthogonalize the predicted column and row spaces. That is, we use column pivoted QR to find orthogonal matrices $\tilde{U} \in \mathbb{R}^{m_1\times s_1}, \tilde{V} \in \mathbb{R}^{m_2\times s_2}$ spanning these spaces. Here $s_1, s_2$ is bounded from above by $2r_n(s+1)+r_G,$ where $r_G$ is the rank of the source $G(t^n).$ 

{\bf Galerkin evolution:} The Galerkin evolution step can be understood as to find an approximation $\hat{X}^{n+1, pre} \in W_{\tilde{U},\tilde{V}},$ such that
\begin{equation}
\label{eqn:galerkin}
\langle \hat{X}^{n+1,pre}, A \rangle = \langle \hat{X}^{n}+ \Delta t F(\hat{X}^{n+1,pre}, t^{n+1}), A \rangle, \qquad \forall A \in W_{\tilde{U},\tilde{V}},
\end{equation}
where $W_{\tilde{U},\tilde{V}}=\{ A \in \mathbb{R}^{m_1\times m_2}: A=\tilde{U} \Sigma \tilde{V}^T, \textrm{with } \Sigma \in \mathbb{R}^{s_1\times s_2} \}$ denote all size $m_1 \times m_2$ spaces with column and row spaces as $\tilde{U}$ and $\tilde{V}.$ 

To solve this problem we let $\hat{X}^{n+1,pre}=\tilde{U} \tilde{\Sigma} \tilde{V}^T, A= \tilde{U}  \Sigma^* \tilde{V}^T.$ Then by \eqref{eqn:innerp}, we get
$$
\langle \tilde{\Sigma}, \Sigma^* \rangle = \langle \tilde{U}^T(\hat{X}^{n}+\Delta t  F(\tilde{U}\tilde{\Sigma}\tilde{V}^T,t^{n+1})))\tilde{V}, \Sigma^* \rangle, \qquad \forall \Sigma^* \in \mathbb{R}^{s_1\times s_2}.
$$
This means 
\begin{eqnarray}
\label{eqn:gebef}
\tilde{\Sigma}&=&\tilde{U}^T(\hat{X}^{n}+\Delta t  F(\tilde{U}\tilde{\Sigma}\tilde{V}^T,t^{n+1})))\tilde{V} \nonumber\\
&=& \tilde{U}^TU^n \Sigma^n (V^n)^T\tilde{V}^T+ \Delta t  \sum_{j=1}^s \tilde{U}^T A_j \tilde{U}\tilde{\Sigma}\tilde{V}^T B_j^T \tilde{V}+ \Delta t  \tilde{U}^T U^n_G \Sigma^n_G (V^n_G)^T\tilde{V}.
\label{eqn:steinbe}
\end{eqnarray} 
The equation \eqref{eqn:steinbe} is a linear matrix equation (generalized Sylvester  equation) for the unknown $\tilde{\Sigma}.$ Since the problems we are considering have low rank solutions, we expect the dimensions of the matrix $\tilde{\Sigma}$ to be small. To solve this equation one can either use a direct solves or an iterative solver (e.g. GMRES) \cite{simoncini2016computational}. %

{\bf Truncation:} The truncation step uses the truncated SVD of $\hat{X}^{n+1}=\mathcal{T}_{\epsilon_2}^{svd} (\tilde{U}\tilde{\Sigma}\tilde{V}^T)= \tilde{U}\mathcal{T}_{\epsilon_2}^{svd}(\tilde{\Sigma})\tilde{V}^T.$ Here the truncation level is chosen as $\epsilon_2 = C_2 \Delta t^2$ to ensure accuracy.

Algorithm \ref{algo:Merge} summarizes the Merge method. 

\begin{algorithm}[htbp]
    \SetKwFunction{isOddNumber}{isOddNumber}
    \SetKwInOut{KwIn}{Input}
    \SetKwInOut{KwOut}{Output}
      \SetKwInOut{KwPara}{Parameters}

    \KwIn{numerical solution at $t^n:$ rank $r_n$ matrix $\hat{X}^n$ in its SVD form $\hat{U}^n \hat{\Sigma}^n (\hat{V}^n)^T$.}
    \KwOut{numerical solution at $t^{n+1}:$  rank $r_{n+1}$ matrix $\hat{X}^{n+1}$ in its SVD form $U^{n+1} \Sigma^{n+1} (V^{n+1})^T.$}
	\KwPara{truncation tolerances $\epsilon_1, \epsilon_2$}  
	{\bf (Merge Prediction).} %
	Compute the truncated SVD of the right hand side: $[\mathcal{U}, \mathcal{S},\mathcal{V}]=\mathcal{T}^{sum}_{\epsilon_1}(F(\hat{X}^n,t^n)),$. 

Compute the BUG prediction spaces $K^{n+1}, L^{n+1}$ according to  Algorithm \ref{algo:dlra}.

Orthogonalize the merged spaces  by column pivoted QR to get the prediction spaces $[\tilde{U},\sim,\sim]= {\sf qr}([[\hat{U}^n, \mathcal{U}, K^{n+1}]),$ 
$[\tilde{V},\sim,\sim]= {\sf qr}([\hat{V}^n, \mathcal{V}^n, L^{n+1}]).$ 
	
{\bf (Galerkin Evolution).}  Find $\tilde{\Sigma}$ by solving $\tilde{\Sigma}=\tilde{U}^T(\hat{X}^{n}+\Delta t  F(\tilde{U}\tilde{\Sigma}\tilde{V}^T,t^{n+1}))\tilde{V}.$
 
 {\bf (Truncation).}  $\hat{X}^{n+1} = \hat{U}^{n+1} \hat{\Sigma}^{n+1} (\hat{V}^{n+1})^T=\tilde{U}\mathcal{T}_{\epsilon_2}^{svd}(\tilde{\Sigma})\tilde{V}^T.$
 \caption{Merge method $t^n \rightarrow t^{n+1}$ }
    \label{algo:Merge}
\end{algorithm}

  We now provide estimates of the computational cost of the different parts of the computation. Here we assume $m_1=m_2=m$, $s$ terms in differential equation and rank $r$. 
\begin{itemize}
\item An evaluation of $AU, (BV)^T$ costs $\mathcal{O}(mr)$ for each of $AU$, $BV$, assuming $A$ and $B$ are sparse.
\item The truncation of the right hand side $\mathcal{T}^{sum}_{\epsilon_1}(F)$ has two components. Two QR orthogonalizations, each costs $\mathcal{O}(2(m(rs)^2-\frac{(rs)^2}{3})$ and one SVD of the core dense matrix, this costs $\mathcal{O}((rs)^3)$.
\item The $K$ and $L$ solves  used to compute the BUG spaces  also require the solution of matrix equations but the dimension of these matrix solves are for a $m \times r$ matrix rather than a $sr \times sr$ matrix. For these equations we have found that applying GMRES in vectorized form works well. It should be noted that this is typically the most expensive part of the solvers. 

\item The cost of solving for the Galerkin core matrix $\tilde{\Sigma}$ in 
\[
\tilde{\Sigma}- \Delta t  \tilde{U}^T F(\tilde{U}\tilde{\Sigma}\tilde{V}^T,t^{n+1})\tilde{V}=\tilde{U}^T\hat{X}^{n}\tilde{V},
\]
depends on the algorithm used.  For example, if GMRES is applied on the vectorized version of this equation each evaluation of the left hand side costs $\mathcal{O}(m(rs)^2)$ if $F$ is a general function. If $F$ is in the form described in equation (\ref{eqn:mode}) it is possible to pre-compute dense $rs \times rs$ matrices (at a cost of $\mathcal{O}(ms(rs)^2)$) and bring down the cost of the evaluation of the left hand side to $\mathcal{O}(s(rs)^3)$ per evaluation.  

Neglecting the cost associated with the growth of the GMRES Krylov subspace this cost would then be multiplied with the number of iterations needed to converge (which is difficult to estimate). We have found that when the governing equation has elliptic terms it also works well to reformulate the Galerkin equation into a fixed point iteration
\[
A_2^{-1}(I- \tilde{U}^T\Delta t A_1\tilde{U}) \tilde{\Sigma}^{k} -  \tilde{\Sigma}^{k} (\Delta t B_2 \tilde{V})^T \tilde{V} B^{-1}_1 = P(\tilde{\Sigma}^{k-1},\hat{X}^{n},\tilde{U},\tilde{V}).
\]  
In the iteration we use a dense Sylvester solver with cost $\mathcal{O}((rs)^3)$ for each iteration. When the matrices $A_1$ and $B_2$ correspond to approximations to second derivatives in the 1 and 2 direction and $A_2$ and $B_1$ are identity or diagonal positive definite matrices we find that this iteration converges to machine precision in a handful of iterations.    
\end{itemize} 

Finally, we would like to comment on the rank evolution of the numerical solution. We find in our numerical experiments that the low rank method generally track the rank growth of the implicit Euler scheme well, i.e. the rank of $\hat{X}^n$ is on par with the rank of the implicit Euler solution at $t^n.$ However, due to the merge in the prediction step, the predicted rank is larger than the rank of the BUG space (which is equal to $2r$). This do translate to larger computational cost in the Galerkin step compared to the BUG solver. However we note that the upper bound in the BUG space $2r$ can be insufficient to capture rapidly changing initial layers in the PDEs as pointed out in  \cite{Nakao:2023aa} (where it was observed a larger than actual rank needs to be imposed for the numerical initial solution). For the numerical experiments we performed in this paper, we find that by merging with the spaces generated by the explicit schemes, this issue seems to be addressed and we do not need to impose large artificial rank for the numerical initial condition.

\subsection{The Merge-adapt method}
As will be shown in Lemma \ref{lemma:lte1}, the column and row spaces generated from the current solution $\hat{U}^n, \hat{V}^n$ combined with $\mathcal{U}$ and $\mathcal{V}$ in $[\mathcal{U}, \mathcal{S},\mathcal{V}]=\mathcal{T}^{sum}_{\epsilon_1}(F(\hat{X}^n,t^n)),$ will be sufficiently accurate to yield a first order accurate solution as long as the time step is sufficiently small compared to the Lipschitz constant $L$ of $\|F\|$. This time step restriction can be overcome by merging with the BUG spaces $K^{n+1}, L^{n+1}$, which are computed useing an implicit solver and can handle the stiffness as shown in the Merge method in Algorithm \ref{algo:Merge}. However, since the computation of $K^{n+1}$ and $L^{n+1}$ requires {solving a linear system}, it is preferable to only add those when necessary. This motivates the design of the Merge-adapt method, which is described in Algorithm \ref{algo:Merge-Adapt}. In Steps 1-4 in Algorithm \ref{algo:Merge-Adapt}, we perform a calculation based purely on  column and row spaces generated by the explicit scheme for the term $\mathcal{T}^{sum}_{\epsilon_1}(F(\hat{X}^n,t^n)).$ In Step 5, we do a residual check, and if this fails, the method will fall back to the Merge method as shown in Steps 6-8. Otherwise, we proceed to the next time step.

 We, heuristically, argue that for this residual check can be used to robustly maintain first order accuracy. For simplicity, assume that the differential equation we would like to solve is 
\[
\frac{d}{dt}X(t) = AX(t)B^T + G(t).
\] 
Let  $X^{n+1}$ be the solution obtained by the classic implicit Euler method, i.e. it solves
\[
{\sf vec}(X^{n+1}) = \underbrace{(I - \Delta t (B \otimes A))^{-1}}_{C} {\sf vec}(X^{n}) + \Delta t  \, {\sf vec}(G(t_{n+1})).
\] 
{Here $\otimes$ denotes Kronecker product and we have introduced the matrix $C$ for notational convenience}   
The low rank solution is not, in general, expected to solve this equation but we may introduce the residual $\hat{R}^{n+1}$ so that 
\[
{\sf vec}(\hat{X}^{n+1}) = \underbrace{(I - \Delta t (B \otimes A))^{-1}}_{C} {\sf vec}(\hat{X}^{n}) + \Delta t  \, {\sf vec}(G(t_{n+1})) + {\sf vec}(\hat{R}^{n+1}).
\] 
Combining the equations we find 
\[
{\sf vec}(\hat{X}^{n+1}) - {\sf vec}(X^{n+1}) = C [ {\sf vec}(\hat{X}^{n}) -{\sf vec}(X^{n})] + {\sf vec}(\hat{R}^{n+1}).
\] 

Now if we assume $\|C\| \le 1$ and the residual check pass, then the error between the implicit Euler solution and the implicit adaptive low rank solution $e^n={\sf vec}(\hat{X}^{n}) -{\sf vec}(X^{n})$ satisfies $\|e^{n+1}\|\le \|e^n\|+ C_2 \Delta t^2.$ This is enough to guarantee that the low rank solution is convergent because the implicit Euler solution is first order in time. It is reasonable to assume that the condition $\|C\| \le 1$ is satisfied for most discretizations of diffusion or advection-diffusion equations \cite{gustafsson2013time}.

\begin{algorithm}[h!]
 \SetKwFunction{isOddNumber}{isOddNumber}
 \SetKwInOut{KwIn}{Input}
 \SetKwInOut{KwOut}{Output}
 \SetKwInOut{KwPara}{Parameters}
 \KwIn{numerical solution at $t^n:$ rank $r_n$ matrix $\hat{X}^n$ in its SVD form $\hat{U}^n \hat{\Sigma}^n (\hat{V}^n)^T$.}
 \KwOut{numerical solution at $t^{n+1}:$  rank $r_{n+1}$ matrix $\hat{X}^{n+1}$ in its SVD form $U^{n+1} \Sigma^{n+1} (V^{n+1})^T.$}
 \KwPara{truncation tolerances $\epsilon_1, \epsilon_2$}  
 {\bf (Cheap Prediction).} Compute a first order prediction of the column and row spaces of $\hat{X}^{n+1}$.  
 Compute the truncated SVD of the right hand side: $[\mathcal{U}, \mathcal{S},\mathcal{V}]=\mathcal{T}^{sum}_{\epsilon_1}(F(\hat{X}^n,t^n))$. 

Orthogonalize by column pivoted QR to get the prediction spaces $[\tilde{U},\sim,\sim]= {\sf qr}([[\hat{U}^n, \mathcal{U}]),$ 
$[\tilde{V},\sim,\sim]= {\sf qr}([\hat{V}^n, \mathcal{V}^n]).$ 
	
{\bf (Galerkin Evolution).}  Find $\tilde{\Sigma}$ by solving $\tilde{\Sigma}=\tilde{U}^T(\hat{X}^{n}+\Delta t  F(\tilde{U}\tilde{\Sigma}\tilde{V}^T,t^{n+1}))\tilde{V}.$
 
 {\bf (Truncation).}  $\hat{X}^{n+1} = \hat{U}^{n+1} \hat{\Sigma}^{n+1} (\hat{V}^{n+1})^T=\tilde{U}\mathcal{T}_{\epsilon_2}^{svd}(\tilde{\Sigma})\tilde{V}^T.$

{\bf (Residual Check).} Compute $\hat{R}^{n+1} = \hat{X}^{n+1} - \hat{X}^{n} -  \Delta t  F(\hat{X}^{n+1},t^{n+1})$. 
If $ \| \hat{R}^{n+1} \| < \epsilon_2$ return the solution $\hat{X}^{n+1}$. If not

{\bf (Merge Predicion).} Compute the BUG prediction spaces $K^{n+1}, L^{n+1}$ according to  Algorithm \ref{algo:dlra}.
Orthogonalize the merged spaces by column pivoted QR: $[\tilde{U},\sim,\sim]= {\sf qr}([[\hat{U}^n, \mathcal{U}, K^{n+1}]),$ 
$[\tilde{V},\sim,\sim]= {\sf qr}([\hat{V}^n, \mathcal{V}^n, L^{n+1}]).$

{\bf (Galerkin Evolution).}  Find $\tilde{\Sigma}$ by solving $\tilde{\Sigma}=\tilde{U}^T(\hat{X}^{n}+\Delta t  F(\tilde{U}\tilde{\Sigma}\tilde{V}^T,t^{n+1}))\tilde{V}.$
 
 {\bf (Truncation).}  $\hat{X}^{n+1} = \hat{U}^{n+1} \hat{\Sigma}^{n+1} (\hat{V}^{n+1})^T=\tilde{U}\mathcal{T}_{\epsilon_2}^{svd}(\tilde{\Sigma})\tilde{V}^T.$ 

 \caption{Merge-adapt method $t^n \rightarrow t^{n+1}$ }
    \label{algo:Merge-Adapt}
\end{algorithm}

We note that the acceleration of Algorithm \ref{algo:Merge-Adapt} over Algorithm \ref{algo:Merge} is based on the assumption that residual check passes in Step 5 and the BUG spaces are not computed. If the residual check fails, extra computation is performed in Steps 3-5 which may incur more computational cost. We observe that, for moderately stiff problems with small $s$, the Merge-adapt method has computational advantages. More discussion is provided in the numerical experiment section.

\subsection{Analysis}
\label{sec:analysis}
In this section, we analyze properties of the proposed methods. %

\subsubsection{Stability}
\label{sec:stability}

First, we consider conservative or dissipative systems for which %
 $\langle F(X, t), X\rangle \le 0, \forall t, X.$ Then it follows that
$$
\frac{d}{dt}\|X\|^2=\langle F(X, t), X\rangle \le 0,
$$
that is, the energy ($L^2$ norm) is monotonically decreasing.

\begin{theorem} 
\label{thm:diss}
If we have  $\langle F(X, t), X\rangle \le 0, \forall t, X ,$ then the numerical solutions from Algorithm \ref{algo:Merge} or \ref{algo:Merge-Adapt} satisfy
$$
\|\hat{X}^{n+1}\| \le \|\hat{X}^{n}\|. 
$$
\end{theorem}
\begin{proof} For Algorithm  \ref{algo:Merge}, first note that $\|X^{n+1,pre}\|=\|\tilde{U}\tilde{\Sigma}\tilde{V}^T\|=\|\tilde{\Sigma}\|.$ By \eqref{eqn:gebef} and \eqref{eqn:innerp},
\begin{eqnarray*}
&&\|\tilde{\Sigma}\|^2=\langle \tilde{\Sigma}, \tilde{\Sigma}\rangle =\langle \tilde{U}^T(\hat{X}^{n}+\Delta t  F(\tilde{U}\tilde{\Sigma}\tilde{V}^T, t^{n+1}))\tilde{V} , \tilde{\Sigma}\rangle \\
&&= \langle  \hat{X}^{n}+\Delta t  F(\tilde{U}\tilde{\Sigma}\tilde{V}^T, t^{n+1}) , \tilde{U}\tilde{\Sigma} \tilde{V}^T\rangle \\
&&\le \langle  \hat{X}^{n} , \tilde{U}\tilde{\Sigma} \tilde{V}^T\rangle =  \langle \hat{X}^{n} , X^{n+1,pre} \rangle.
\end{eqnarray*}
Therefore, by Cauchy-Schwarz inequality, $\|X^{n+1,pre}\|^2 \le  \|\hat{X}^{n} \| \|X^{n+1,pre}\|,$ which gives  $\|X^{n+1,pre}\| \le  \|\hat{X}^{n} \|.$ Finally, because of the property of the truncated SVD, $\|\hat{X}^{n+1}\|=\|\mathcal{T}_{\epsilon_2}^{svd} (X^{n+1,pre})\| \le \|X^{n+1,pre}\| \le  \|\hat{X}^{n} \|.$ The proof for Algorithm   \ref{algo:Merge-Adapt} is the same and is omitted.
\end{proof}

We can easily generalize this result to semi-bounded operator \cite{gustafsson2013time}.
Specifically, if $\langle F(X, t), X\rangle \le \alpha \|X\|^2, \forall \, t, X,$ then
$$
\frac{d}{dt}\|X\|^2=\langle F(X, t), X\rangle \le \alpha \|X\|^2,
$$
which implies $\|X(t)\| \le e^{\alpha t} \|X(0)\|.$
\begin{theorem}[Stability for semi-bounded operator] 
\label{thm:stab}
If we have $\langle F(X, t), X\rangle \le \alpha \|X\|^2, \forall \, t, X$ then the numerical solutions from Algorithm \ref{algo:Merge} or \ref{algo:Merge-Adapt} satisfy 
$$
\|\hat{X}^{n+1}\| \le e^{\alpha \Delta t} \|\hat{X}^{n}\|, \quad \|\hat{X}^{n}\| \le e^{\alpha \, t^n} \|\hat{X}^{0}\|, \quad \textrm{if} \,\, \alpha \Delta t \le 1.
$$
\end{theorem}
\begin{proof} The proof is similar to the proof of Theorem \ref{thm:diss}, so we only highlight the difference.
We have 
$$\langle \hat{X}^{n+1}, \hat{X}^{n+1} \rangle \le  \langle \hat{X}^{n}, \hat{X}^{n+1} \rangle + \Delta t \alpha  \|\hat{X}^{n+1}\|^2,$$
which implies $$\|\hat{X}^{n+1}\| \le \frac{1}{1-\alpha \Delta t} \|\hat{X}^{n}\|,$$
if $\alpha \Delta t \le 1.$ The theorem follows using a similar arguments as in Theorem \ref{thm:diss}.
\end{proof}

\subsubsection{Convergence}
A convergence estimate has been shown in Theorem 2 of \cite{ceruti2022rank} for the rank adaptive BUG scheme. Assuming  the Lipschitz continuity and boundedness of the operator $F,$ the authors show that the BUG scheme has numerical error bounded by the sum of the initial numerical error, the tangent projection error, $\epsilon_2$ and first order in time error. 
The  proof is based on the time continuous version of the BUG schemes. We want to point out that for PDE applications, in general the (differential) operators are  not Lipschitz bounded. Nevertheless, it is still of theoretical interests to investigate the convergence properties of the schemes under such assumptions.
 
 For simplicity, below we will estimate the local truncation error of the first order implicit schemes with column and row space spanned by the cheap prediction space, i.e. we are investigating the Steps 1-4 in Algorithm \ref{algo:Merge-Adapt}. Since the Merge method uses a space that is union of the cheap prediction (as shown below) and the BUG space, the local truncation error of Algorithm \ref{algo:Merge} will be also be upper bounded by the estimate below.
\begin{lemma}[Local truncation error with cheap prediction space]
\label{lemma:lte1} Suppose $F$ is Lipschitz-continuous in both variables and bounded, i.e. there exists a constant $L$ such that $\|F(Y, s)-F(Z, t)\| \le L \|Y-Z\| +L|s-t|$ and $\|F(X, t)\| \le B$. 
If we denote $\mathcal{X}$ as the exact solution to \eqref{eqn:mode} with initial condition  $\hat{X}^n=U^n \Sigma^n (V^n)^T$ from $t^n$ to $t^{n+1}$,  $\hat{X}^{n+1}$ the numerical solution obtained from from Algorithm  \ref{algo:Merge-Adapt}, Steps 1-4 with the same initial condition  $\hat{X}^n,$ then for sufficiently small $\Delta t,$ we have the following error bound
\begin{equation}
\label{eqn:lte2}
\|\mathcal{X}-\hat{X}^{n+1}\| \le   C \Delta t^2 + 2 \Delta t \epsilon_1+  \epsilon_2.
\end{equation}
In %
 \eqref{eqn:lte2}, $C$ is a constant that only depends $L$ and $B$.
\end{lemma}

\begin{proof}
First we prove \eqref{eqn:lte2} for $\epsilon_1 = 0$. Then $\mathcal{T}^{sum}_{\epsilon_1=0}(F(\hat{X}^n,t^n)) = F(\hat{X}^n,t^n)$. From \eqref{eqn:gebef},  and left multiplying by $\tilde{U}$ and right multiplying by $\tilde{V}^T,$ we obtain
$$\hat{X}^{n+1,pre}= \tilde{U}\tilde{\Sigma}\tilde{V}^T=\tilde{U}\tilde{U}^T(\hat{X}^{n}+\Delta t F(\hat{X}^{n+1,pre},t^{n+1})))\tilde{V}\tilde{V}^T.
$$
Since the column space of $\hat{X}^n$ is a subset of the column space of $\tilde{U},$ i.e. $R(\hat{X}^n)\subseteq R(\tilde{U}),$ we have $\tilde{U}\tilde{U}^T\hat{X}^{n}=\hat{X}^{n}.$ Similarly, because $R((\hat{X}^n)^T) \subseteq R(\tilde{V}),$ $\hat{X}^{n}\tilde{V}\tilde{V^T}=\hat{X}^{n}.$ Therefore, $\tilde{U}\tilde{U}^T\hat{X}^{n}\tilde{V}\tilde{V^T}=\hat{X}^{n}.$ By a similar argument, $R(F(\hat{X}^n,t^n))\subseteq R(\tilde{U})$ and $R(F(\hat{X}^n,t^n)^T) \subseteq R(\tilde{V})$ implies $\tilde{U}\tilde{U}^TF(\hat{X}^{n},t^{n})\tilde{V}\tilde{V}^T=F(\hat{X}^{n},t^{n}).$ This gives
\begin{eqnarray}
\label{eqn:cc}
&&\hat{X}^{n+1,pre}=\hat{X}^{n} +\Delta t \tilde{U}\tilde{U}^TF(\hat{X}^{n+1,pre},t^{n+1})\tilde{V}\tilde{V}^T\notag \\
&&=\hat{X}^{n} +\Delta t \tilde{U}\tilde{U}^TF(\hat{X}^{n},t^{n})\tilde{V}\tilde{V}^T+\Delta t \tilde{U}\tilde{U}^T(F(\hat{X}^{n+1,pre},t^{n+1})-F(\hat{X}^{n},t^{n}))\tilde{V}\tilde{V}^T \notag \\
&&=\hat{X}^{n} +\Delta t  F(\hat{X}^{n},t^{n})+\Delta t \tilde{U}\tilde{U}^T(F(\hat{X}^{n+1,pre},t^{n+1})-F(\hat{X}^{n},t^{n}))\tilde{V}\tilde{V}^T. 
\end{eqnarray}
Therefore,
\begin{multline}
\label{eqn:pf1}
\|\hat{X}^{n+1,pre}-\mathcal{X}\| \\
=\|\hat{X}^{n} -\mathcal{X} +\Delta t  F(\hat{X}^{n},t^{n})+\Delta t \tilde{U}\tilde{U}^T(F(\hat{X}^{n+1,pre},t^{n+1})-F(\hat{X}^{n},t^{n}))\tilde{V}\tilde{V}^T\|.  
\end{multline}
We have $ \hat{X}^{n} -\mathcal{X} =-\Delta t F(X(t^*),t^*),$ where $t^*$ is a point on the interval $[t^n, t^{n+1}].$ Therefore,
\begin{eqnarray}
\label{eqn:pf2}
&&\|\hat{X}^{n} -\mathcal{X} +\Delta t  F(\hat{X}^{n},t^{n})\|=\Delta t \|- F(X(t^*),t^*) +  F(\hat{X}^{n},t^{n})\| \notag \\
&&\le \Delta t L(\Delta t+ \|X(t^*)-\hat{X}^{n}\|)  \le \Delta t^2 L(B+1). 
\end{eqnarray}
On the other hand,
\begin{eqnarray*}
&&\|\tilde{U}\tilde{U}^T(F(\hat{X}^{n+1,pre},t^{n+1})-F(\hat{X}^{n},t^{n}))\tilde{V}\tilde{V}^T \|\\
&& \le  \|\tilde{U}\tilde{U}^T\|_2 \|F(\hat{X}^{n+1,pre},t^{n+1})-F(\hat{X}^{n},t^{n})\| \|\tilde{V}\tilde{V}^T\|_2,
\end{eqnarray*}
where $\|\cdot\|_2$ is the matrix 2-norm. Since $ \|\tilde{U}\tilde{U}^T\|_2= \|\tilde{V}\tilde{V}^T\|_2=1$ and
$ 
\|F(\hat{X}^{n+1,pre},t^{n+1})-F(\hat{X}^{n},t^{n})\| \le L(\Delta t + \|\hat{X}^{n}-\mathcal{X}\| +\|\hat{X}^{n+1,pre}-\mathcal{X}\|) \le L(\Delta t + B \Delta t +\|\hat{X}^{n+1,pre}-\mathcal{X}\| ), 
$ we get
\begin{eqnarray}
\label{eqn:pf3}
 \|\tilde{U}\tilde{U}^T(F(\hat{X}^{n+1,pre},t^{n+1})-F(\hat{X}^{n},t^{n}))\tilde{V}\tilde{V}^T \|
 \le  L(\Delta t + B \Delta t+\|\hat{X}^{n+1,pre}-\mathcal{X}\| ).
\end{eqnarray}
Combining \eqref{eqn:pf1},\eqref{eqn:pf2},\eqref{eqn:pf3}, we get
$$
\|\hat{X}^{n+1,pre}-\mathcal{X}\| \le \Delta t^2 L(2B+2) + L\Delta t \|\hat{X}^{n+1,pre}-\mathcal{X}\|,
$$
Choosing $\Delta t$ small enough (i.e. $\Delta t \le \frac{1}{2L}$), we proved
\begin{eqnarray}
\label{eqn:pf4}
\|\hat{X}^{n+1,pre}-\mathcal{X}\| \le 4\Delta t^2 L(B+1).
\end{eqnarray}
Therefore, \eqref{eqn:lte2} follows by using $\|\hat{X}^{n+1,pre}-\hat{X}^{n+1}\|=\|\hat{X}^{n+1,pre}-\mathcal{T}_{\epsilon_2}^{svd}  (\hat{X}^{n+1,pre})\| \le \epsilon_2.$

The proof for \eqref{eqn:lte2} with $\epsilon_1 \neq 0$ largely follows the same process. The only difference is that we no longer have $\tilde{U}\tilde{U}^TF(\hat{X}^{n},t^{n})\tilde{V}\tilde{V}^T=F(\hat{X}^{n},t^{n})$ due to the truncation in the prediction step. Instead, we have $\tilde{U}\tilde{U}^T\mathcal{T}_{\epsilon_1} (F(\hat{X}^{n},t^{n}))\tilde{V}\tilde{V}^T=\mathcal{T}_{\epsilon_1} (F(\hat{X}^{n},t^{n})).$ In this case, 
$$\tilde{U}\tilde{U}^TF(\hat{X}^{n},t^{n})\tilde{V}\tilde{V}^T=\mathcal{T}_{\epsilon_1} (F(\hat{X}^{n},t^{n}))+\tilde{U}\tilde{U}^T(F(\hat{X}^{n},t^{n})-\mathcal{T}_{\epsilon_1} (F(\hat{X}^{n},t^{n})))\tilde{V}\tilde{V}^T,$$
and because $ \|\tilde{U}\tilde{U}^T\|_2= \|\tilde{V}\tilde{V}^T\|_2=1,$ this leads to 
$$
\|\tilde{U}\tilde{U}^TF(\hat{X}^{n},t^{n})\tilde{V}\tilde{V}^T-F(\hat{X}^{n},t^{n})\| \le 2 \|F(\hat{X}^{n},t^{n})-\mathcal{T}_{\epsilon_1} (F(\hat{X}^{n},t^{n}))\|=2 \epsilon_1.
$$
By a similar argument,  \eqref{eqn:lte2} follows. The details are omitted for brevity.
\end{proof}

The numerical convergence  follows from Lemma \ref{lemma:lte1}  by a standard result for the convergence of one-step method for ODE and is skipped for brevity. Now we would like to comment on the results in this Lemma to understand property of the scheme. As we can see from the results in Section \ref{sec:stability}, the schemes are unconditionally stable due to the Galerkin step. Also, in Lemma \ref{lemma:lte1}, the tangent projection error is absent because the DLRA approximation is not used. However, the time step is limited by the Lipschitz constant and also the constant in the estimate $C$ depends on the Lipschitz constant. For  standard error estimates for implicit schemes for stiff problems, the concept of B-convergence \cite{frank1981concept} can remove such restrictions with the help of one-sided Lipschitz continuity. Unfortunately, such results are not available in our case.

\section{Numerical results}
\label{sec:numerical}
We now present numerical results that illustrate the features of the methods described above. In this section we will refer to the different methods in figures and tables according to the following naming convention. We will denote the Merge method by M, and the Merge-adapt method by MA. For many problems we will also compare to the classic implicit Euler discretization as IE or Implicit Euler. When computing errors we will compare to a solution computed using Matlab's built in {\tt ODE15s} or {\tt ODE45} solvers with relative and absolute tolerance set to $10^{-12}$. We always report relative errors in the  Frobenius norm. We note that in all the examples below we use $\epsilon_1=0$ but that the results are very similar for $\epsilon_1 = \Delta t$.  { Additionally, we adjust the truncation $\epsilon_2$ to the PDE setting and chose $\epsilon_2 =  \mathcal{O} ((\Delta t^2+h_1^3+h_2^3)/\sqrt{h_1 h_2}). $ Here the first term is from the time stepping local truncation error, the second term is from the spatial discretization error and the denominator is rescaling for the norms to be independent of grid spacing.} {In some of the results we also include the ``hyperbolic and parabolic CFL numbers'', ($\Delta t/h$  and $\Delta t/h^2$) to illustrate that the method can take large timesteps when parabolic terms are present while still achieving temporal accuracy.   } 

The main PDE examples used in this section are simulations of solutions to the equation 
\begin{multline} \label{eq:PDE_num_example}
\frac{\partial \rho}{\partial t} + r_1(x_1) \frac{\partial \rho}{\partial x_2} + r_2(x_2) \frac{\partial \rho}{\partial x_1} = \\
 b_1(x_2) \frac{\partial}{\partial x_1} \left[ a_{1}(x_1) \frac{\partial \rho}{\partial x_1} \right]  
 + b_2(x_2) \frac{\partial^2 \left[ a_{2}(x_1) \rho \right]}{\partial x_1 \partial x_2}  
+ a_3(x_1) \frac{\partial^2 \left[ b_{3}(x_2) \rho \right]}{\partial x_1 \partial x_2}  
  + a_4(x_1) \frac{\partial}{\partial x_2} \left[ b_{4}(x_2) \frac{\partial \rho}{\partial x_2} \right].
\end{multline}
Here $\rho = \rho(t,x_1,x_2)$ and we always use homogenous Dirichlet boundary conditions. Equation (\ref{eq:PDE_num_example}) is discretized on the domain $(x_1,x_2) \in [-1,1]^2,$ using second order accurate finite differences. Consider the grid $(x_{1,i}, x_{2,j}) = (-1 + i h_1, -1 + j h_2)$, with $h_1 = 2/(m_1+1),$ $h_2 = 2/(m_2+1).$ The solution $\rho(t_n,x_{1,i},x_{2,j})$ is then approximated by the grid function (matrix) $\hat{X}^{n}_{i,j}$ and derivatives are approximated as follow:
\[
 \frac{\partial}{\partial x_1} \left[ a_{1}(x_1) \frac{\partial \rho(t_n,x_{1},x_{2})}{\partial x_1} \right] \Bigg|_{x_{1,i},x_{2,j}}   \approx D_+^1 \left[ \frac{a_1(x_{1,i})+a_1(x_{1,i-1})}{2} \right]D_-^1\, X_{i,j}^n, 
\] 
\[
 \frac{\partial^2 \left[ a_{2}(x_1) \rho \right]}{\partial x_1 \partial x_2}  \Bigg|_{x_{1,i},x_{2,j}}  =  \frac{\partial}{\partial x_1} \left[ a_{2}(x_1) \frac{\partial \rho(t_n,x_{1},x_{2})}{\partial x_2} \right] \Bigg|_{x_{1,i},x_{2,j}}   \approx D_0^1  a_2(x_{1,i}) D_0^2\, X_{i,j}^n. 
\] 
with the remaining two terms discretized in the same way. Here the difference operators  in the 1-direction are defined 
\[
2h_1 D_0^1 w_{i,j} \equiv w_{i+1,j} - w_{i-1,j}, \ \ h_1 D_+^1 w_{i,j} \equiv w_{i+1,j} - w_{i,j},  \ \ h_1 D_-^1 w_{i,j} \equiv w_{i,j} - w_{i-1,j}, 
\]
with the operators in the 2-direction defined analogously. 

{The equation (\ref{eq:PDE_num_example}) thus take the semidiscrete form
\begin{multline}
\frac{d}{dt} X + R_1 X (D_0^{(2)})^T + D_0^{(1)} X (R_2)^T = \\ 
L_{a_1} X B_1^T + D_0^{(1)} A_2 X (D_0^{(2)})^T B_2^T   + A_3 D_0^{(1)} X B_3^T (D_0^{(2)})^T + A_4 X L_{b_4}^T.
\end{multline}
Here $R_1, R_2, B_1,  A_2, B_2,  A_3, B_4, A_4$ are diagonal matrices with $r_1(x_1), r_2(x_2), b_1(x_2), $ etc. on the diagonal. Their sizes are $n_1 \times n_1$ or $n_2 \times n_2$ depending on if they appear on the left or right of $X$. The matrices $D_{0}^{(k)}, \, k = 1,2,$ are matrices with super-diagonal $0.5/h_k$ and sub-diagonal $-0.5/h_k$. The matrices $L_{a_1}$ and $L_{b_4}$ have elements $(L_{a_1})_{i,i-1} = (a_1(x_{1,i})+a_1(x_{1,i-1}))/(2h_1^2)$, $(L_{a_1})_{i,i-1} = (a_1(x_{1,i+1})+2 a_1(x_{1,i})+a_1(x_{1,i-1}))/(4h_1^2)$, and $(L_{a_1})_{i,i+1} = (a_1(x_{1,i+1})+a_1(x_{1,i}))/(2h_1^2),$ (with  $L_{b_4}$) defined similarly.
}

\subsection{Solid body rotation} \label{rot_example}
In this experiment we consider a hyperbolic problem with pure solid body rotation, corresponding to, $r_1(x_1) = x_1 $ $r_2(x_2) = -x_2$ and $a_j = b_j = 0,\, j = 1,\ldots,4$. We start from the rank-1 initial data 
\[
\rho(0,x_1,x_2) = e^{-\left(\frac{x_1}{0.3}\right)^2} e^{-\left(\frac{x_2}{0.1}\right)^2},
\] 
and evolve the solution until time $t = \pi$ using a base timestep of $\Delta t = \pi / n_T$, with $n_T=40, 80, 160, 320$. We carry out the computation on three different grids with $m_1=m_2 = 99, 199, 799$.  The errors are listed in Table \ref{tab:rot_errors}. Again similar conclusions can be drawn, namely all three methods give comparable errors. The rank evolution is reported in Figure \ref{fig:pure_rot_rank}. In this problem, the rank evolution exhibits periodic pattern in time due to the solution structure. The rank of the low rank methods are lower than the implicit Euler method.

In Figure \ref{fig:rot_contour} we display the solution computed using a grid with $m_1=m_2 = 99$. Results using the Merge method (left), the very accurate reference solution (middle) and the BUG method (right) are presented. The solutions are displayed at time $0.5\pi$ when they have rotated from being horizontal to vertical. The BUG method does not see the rotation (it is outside the tangent space) and remains stationary. 
\begin{figure}[htb]
\begin{center}	
\includegraphics[width=0.32\textwidth,trim={0.0cm 0.0cm 0.0cm 0.0cm},clip]{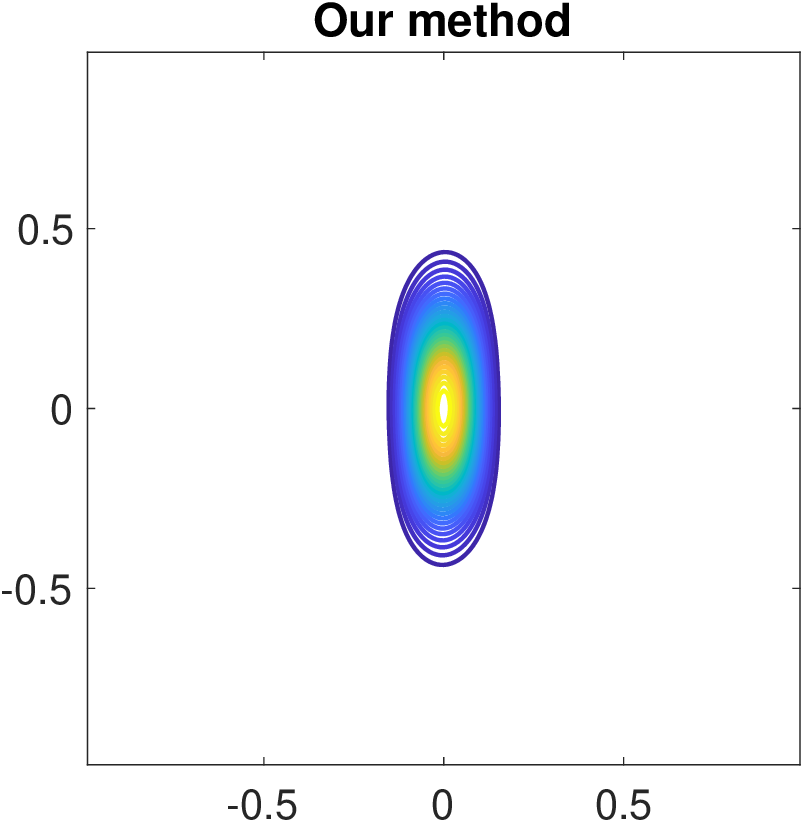}
\includegraphics[width=0.32\textwidth,trim={0.0cm 0.0cm 0.0cm 0.0cm},clip]{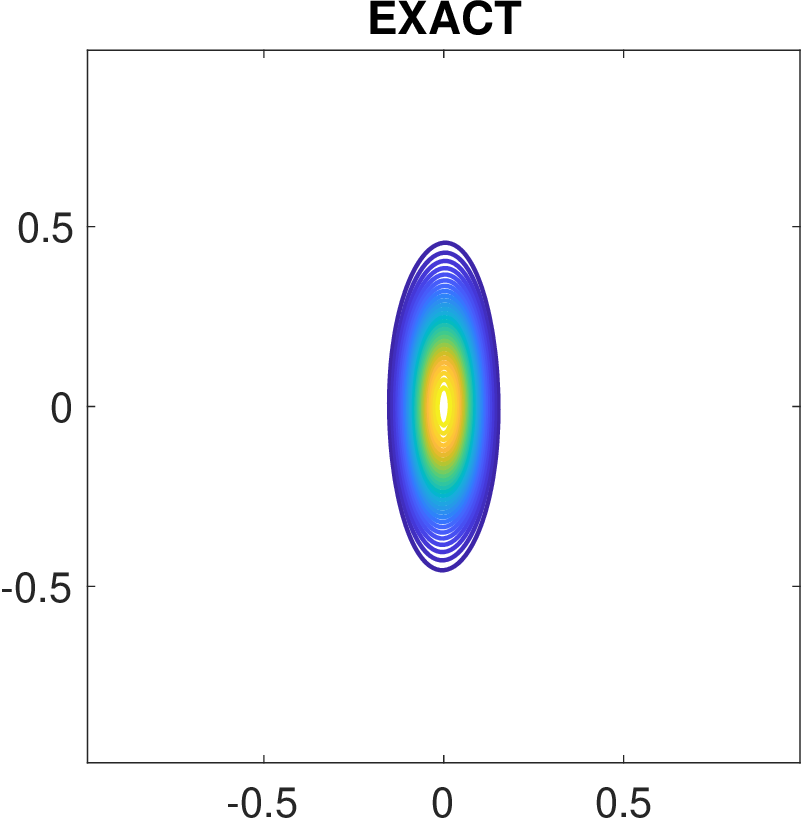}
\includegraphics[width=0.32\textwidth,trim={0.0cm 0.0cm 0.0cm 0.0cm},clip]{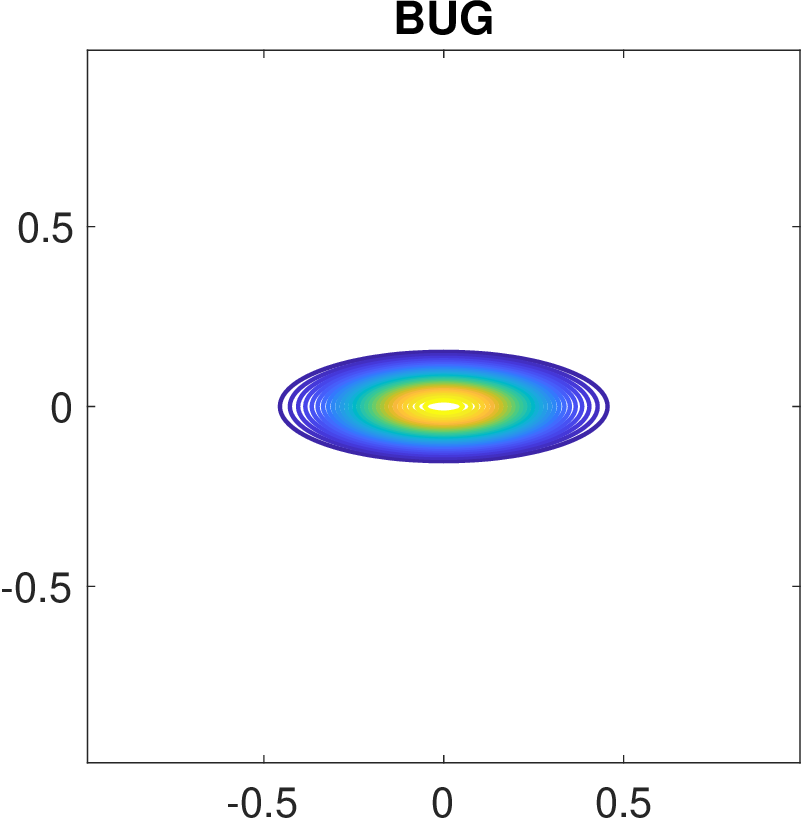}
\caption{The solution to the solid body rotation problem computed using the Merge method (left), the a highly accurate reference solution (middle) and the BUG method (right). The BUG method does not see the rotation (it is outside the tangent space) and remains stationary. \label{fig:rot_contour}}
\end{center}
\end{figure}

 \begin{table}[h!] 
 \begin{center} 
  \begin{tabular}{| p{0.6cm} | p{2.5cm} | p{4.0cm} | p{2.8cm}  |  p{2.2cm}  |}
 \hline
    $n_T $ & M & MA & IE & $\Delta t/h$  \\
\hline 
\hline 
40 &  2.50(-1) &  2.50(-1) F = 40&  2.51(-1) &    3.9 :      0 \\
\hline 
80 &  1.71(-1)  [0.54] &  1.71(-1)  [0.54] F = 80&  1.73(-1)  [0.53] &      2 :      0 \\
\hline 
160 &  1.15(-1)  [0.57] &  1.15(-1)  [0.57] F = 155&  1.10(-1)  [0.64] &   0.98 :      0 \\
\hline 
320 &  7.12(-2)  [0.69] &  7.11(-2)  [0.69] F = 307&  6.60(-2)  [0.74] &   0.49 :      0 \\
\hline 
\hline 
40 &  2.50(-1) &  2.50(-1) F = 32&  2.54(-1) &    7.9 :      0 \\
\hline 
80 &  1.75(-1)  [0.51] &  1.75(-1)  [0.51] F = 78&  1.76(-1)  [0.53] &    3.9 :      0 \\
\hline 
160 &  1.14(-1)  [0.62] &  1.14(-1)  [0.62] F = 152&  1.13(-1)  [0.63] &      2 :      0 \\
\hline 
320 &  6.92(-2)  [0.72] &  6.92(-2)  [0.72] F = 299&  6.81(-2)  [0.73] &   0.98 :      0 \\
\hline 
\hline 
40 &  2.51(-1) &  2.51(-1) F = 39&  2.55(-1) &     31 :      0 \\
\hline 
80 &  1.75(-1)  [0.52] &  1.75(-1)  [0.52] F = 74&  1.77(-1)  [0.52] &     16 :      0 \\
\hline 
160 &  1.13(-1)  [0.62] &  1.13(-1)  [0.62] F = 139&  1.14(-1)  [0.63] &    7.9 :      0 \\
\hline 
320 &  6.86(-2)  [0.72] &  6.86(-2)  [0.72] F = 298&  6.88(-2)  [0.73] &    3.9 :      0 \\
\hline 
 \end{tabular}
 \caption{Solid body rotation. Displayed are the errors for different timesteps along with estimated rates of convergence (in brackets). The top box is for $m_1=m_2=99$, the middle for $m_1=m_2=199$ and the bottom for $m_1=m_2=799$. The F indicates how many times the Merge-adapt needed to add the BUG spaces. \label{tab:rot_errors}}
 \end{center} 
 \end{table} 
 
 \begin{figure}[htb!]
\begin{center}	
\includegraphics[width=0.49\textwidth,trim={0.0cm 0.0cm 0.0cm 0.0cm},clip]{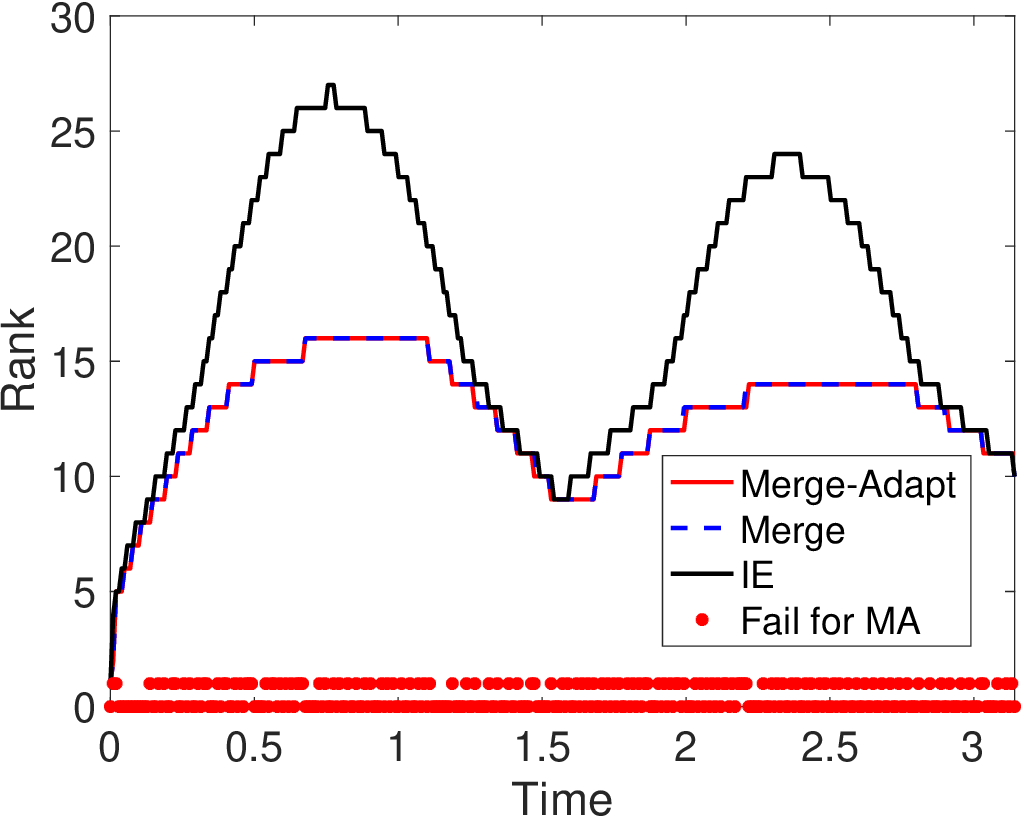}
\caption{Displayed is the {truncated} rank for the Merge, Merge-adapt and Implicit Euler method for the problem with solid body rotation problem. Here ``Fail for MA'' is the indicator where the BUG space is needed for the Merge-adapt method. For the Implicit Euler method we constantly use the threshold $\Delta t^2$ when computing the rank via the truncated SVD. All the computations are done with $m_1 = m_2 = 799$ and $n_{T} = 320.$ \label{fig:pure_rot_rank}}
\end{center}
\end{figure}

\subsection{Rotation with anisotropic diffusion} \label{ani_rot_example}
In this experiment we consider solid body rotation, $r_1(x_1) = x_1 $ $r_2(x_2) = -x_2$, with anisotropic diffusion. The diffusion coefficients are:
\begin{eqnarray*}
&&a_1(x_1) = a_4(x_1) =  \sqrt{\mu}(1 + 0.1 \sin(\pi x_1)), \\ 
&&a_2(x_1) =\sqrt{\mu}(0.15 + 0.1 \sin(\pi x_1)), \\
&&a_3(x_1) =\sqrt{\mu}(0.15 + 0.1 \cos(\pi x_1)), \\ 
&&b_1(x_2) = b_4(x_2) = \sqrt{\mu}(1 + 0.1 \cos(\pi x_2)), \\ 
&&b_2(x_2) = \sqrt{\mu}(0.15 + 0.1 \cos(\pi x_2)), \\
&&b_3(x_2) = \sqrt{\mu}(0.15 + 0.1 \sin(\pi x_2)),  
\end{eqnarray*}
and   $\mu = 10^{-3}$. We start from the rank-1 initial data 
\[
\rho(0,x_1,x_2) = e^{-\left(\frac{x_1}{0.3}\right)^2} e^{-\left(\frac{x_2}{0.1}\right)^2},
\] 
and evolve the solution until time $t = \pi$ using a timestep  $\Delta t = \pi / n_T$, with $40, 80, 160, 320$. We carry out the computation on three different grids with $m_1=m_2 = 99, 199, 799$.  

\begin{table}[h!] 
 \begin{center} 
  \begin{tabular}{| p{0.6cm} | p{2.9cm} | p{4.2cm} | p{2.6cm}  |  p{3.0cm}  |}
 \hline
    $n_T $ & M & MA & IE & $\Delta t/h$ : $\mu \Delta t/h^2$  \\
\hline 
\hline 
40 &  1.65(-1) &  1.65(-1) F = 39&  1.60(-1) &    3.9 :    0.2 \\
\hline 
80 &  1.15(-1)  [0.51] &  1.15(-1)  [0.51] F = 78&  1.01(-1)  [0.65] &      2 :  0.098 \\
\hline 
160 &  6.88(-2)  [0.74] &  6.88(-2)  [0.74] F = 155&  6.01(-2)  [0.75] &   0.98 :  0.049 \\
\hline 
320 &  4.38(-2)  [0.65] &  4.38(-2)  [0.65] F = 292&  3.36(-2)  [0.83] &   0.49 :  0.025 \\
\hline 
\hline 
40 &  1.62(-1) &  1.62(-1) F = 39&  1.60(-1) &    7.9 :   0.79 \\
\hline 
80 &  1.03(-1)  [0.64] &  1.03(-1)  [0.64] F = 74&  1.02(-1)  [0.65] &    3.9 :   0.39 \\
\hline 
160 &  6.26(-2)  [0.72] &  6.26(-2)  [0.72] F = 133&  6.07(-2)  [0.75] &      2 :    0.2 \\
\hline 
320 &  3.59(-2)  [0.80] &  3.59(-2)  [0.80] F = 288&  3.40(-2)  [0.83] &   0.98 :  0.098 \\
\hline 
\hline 
40 &  1.61(-1) &  1.61(-1) F = 38&  1.61(-1) &     31 :     13 \\
\hline 
80 &  1.02(-1)  [0.65] &  1.02(-1)  [0.65] F = 77&  1.02(-1)  [0.65] &     16 :    6.3 \\
\hline 
160 &  6.10(-2)  [0.75] &  6.10(-2)  [0.75] F = 138&  6.09(-2)  [0.75] &    7.9 :    3.1 \\
\hline 
320 &  3.43(-2)  [0.82] &  3.43(-2)  [0.82] F = 273&  3.41(-2)  [0.83] &    3.9 :    1.6 \\
\hline 
 \end{tabular}
 \caption{Solid body rotation with anisotropic diffusion. Displayed are the errors (here 1.3(-1) means $1.3 \cdot 10^{-1}$) for different timesteps along with estimated rates of convergence (in brackets). The numbers on the far right are the ``hyperbolic and parabolic CFL numbers''. The top box is for $m_1=m_2=99$, the middle for $m_1=m_2=199$ and the bottom for $m_1=m_2=799$. The F indicates how many times the Merge-adapt needed to add the BUG spaces. \label{tab:ani_rot_errors}} 
 \end{center} 
 \end{table}

In Table \ref{tab:ani_rot_errors} we display the errors and rates of convergence for the Merge, Merge-adapt and the implicit Euler method. All three methods behave almost identically (the errors and rates of convergence only differ in the third or fourth digit). However, the time it takes to compute the solution is quite different. 

We also track the rank of the solution as a function of time.  In Figure \ref{fig:ani_rot_rank} we display the rank for the different methods as a function of time. Once the solution becomes more diagonal at time $\pi/4$ the rank is decreased and then increased again as the rotation continues. For the Merge and Merge-adapt method we use the truncation tolerance $\epsilon_2 = \Delta t^2$ and in order to compare with the classic implicit Euler we compute the SVD and count the rank of that solution using the same threshold. We observe the ranks from the low rank methods are slightly smaller than the implicit Euler method.

\begin{figure}[htb!]
\begin{center}	
\includegraphics[width=0.49\textwidth,trim={0.0cm 0.0cm 0.0cm 0.0cm},clip]{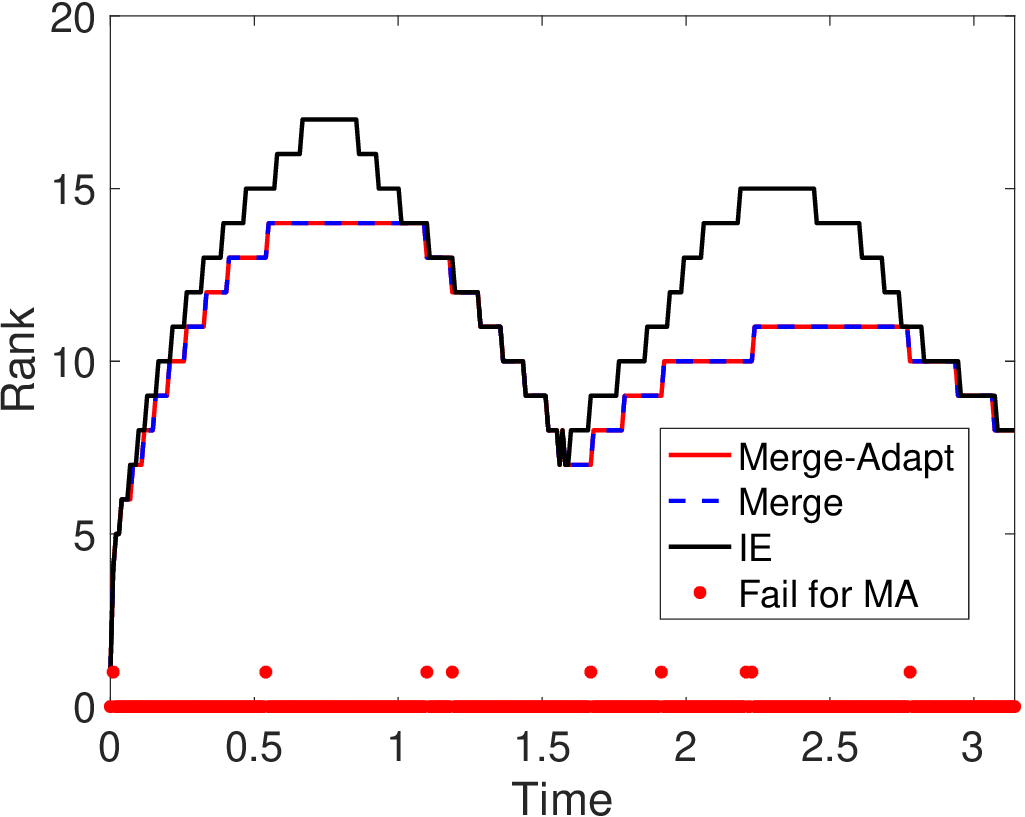}
\caption{Displayed is the {truncated} rank for the Merge, Merge-adapt and Implicit Euler method for the problem with rotation and anisotropic diffusion. Here ``Fail for MA'' is the indicator where the BUG space is needed for the Merge-adapt method. For the Implicit Euler method we constantly use the threshold $\Delta t^2$ when computing the rank via the truncated SVD. All the computations are done with $m_1 = m_2 = 799$ and $n_{T} = 320.$ \label{fig:ani_rot_rank}}
\end{center}
\end{figure}

\subsection{Anisotropic diffusion} \label{ani_heat_example}
In this experiment we consider anisotropic diffusion without rotation ($r_1 = r_2 = 0$). We take the diffusion coefficients to be constants with $a_1 = a_4 = b_1 = b_4 = 1$ and $a_2 = a_3 = b_2 = b_3 = 0.3$.  Starting from the rank-1 initial data 
\[
\rho(0,x_1,x_2) = \sin( \pi x_1) \sin(\pi x_2),
\] 
we evolve the solution until time $t = 0.5$ using a base timestep of $\Delta t = 0.5 / n_T$, with $40, 80, 160, 320, 640, 1280$. We carry out the computation on three different grids with $m_1=m_2 = 99, 399, 799$.  

 \begin{table}[h!] 
 \begin{center} 
  \begin{tabular}{| p{0.8cm} | p{2.5cm} | p{4.2cm} | p{2.8cm}  |  p{2.5cm}  |}
 \hline
    $n_T $ & M & MA & IE & $\Delta t/h$ : $\Delta t/h^2$  \\
\hline 
\hline 
40 &  9.33(-2) &  9.34(-2) F = 34&  9.31(-2) &   0.62 :     31 \\
\hline 
80 &  3.05(-2)  [1.61] &  3.20(-2)  [1.54] F = 58&  4.39(-2)  [1.08] &   0.31 :     16 \\
\hline 
160 &  1.06(-2)  [1.52] &  1.47(-2)  [1.11] F = 99&  2.13(-2)  [1.03] &   0.16 :    7.8 \\
\hline 
320 &  4.85(-3)  [1.12] &  1.04(-2)  [0.50] F = 172&  1.05(-2)  [1.01] &  0.078 :    3.9 \\
\hline 
640 &  3.24(-3)  [0.58] &  5.42(-3)  [0.94] F = 325&  5.22(-3)  [1.01] &  0.039 :      2 \\
\hline 
1280 &  2.05(-3)  [0.66] &  2.35(-3)  [1.20] F = 513&  2.58(-3)  [1.01] &   0.02 :   0.98 \\
\hline 
\hline 
40 &  1.17(-1) &  1.17(-1) F = 39&  9.30(-2) &    1.2 :  1.2e+02 \\
\hline 
80 &  4.50(-2)  [1.37] &  4.50(-2)  [1.37] F = 71&  4.38(-2)  [1.08] &   0.62 :     62 \\
\hline 
160 &  1.61(-2)  [1.48] &  1.66(-2)  [1.44] F = 123&  2.13(-2)  [1.03] &   0.31 :     31 \\
\hline 
320 &  4.97(-3)  [1.69] &  6.72(-3)  [1.30] F = 221&  1.05(-2)  [1.01] &   0.16 :     16 \\
\hline 
640 &  3.01(-3)  [0.72] &  5.79(-3)  [0.21] F = 409&  5.22(-3)  [1.01] &  0.078 :    7.8 \\
\hline 
1280 &  2.29(-3)  [0.39] &  3.08(-3)  [0.91] F = 654&  2.58(-3)  [1.01] &  0.039 :    3.9 \\
\hline 
\hline 
40 &  1.58(-1) &  1.58(-1) F = 40&  9.30(-2) &      5 :  2000 \\
\hline 
80 &  6.87(-2)  [1.20] &  6.87(-2)  [1.20] F = 80&  4.38(-2)  [1.08] &    2.5 :  1000 \\
\hline 
160 &  2.75(-2)  [1.32] &  2.75(-2)  [1.32] F = 160&  2.13(-2)  [1.03] &    1.2 :  500 \\
\hline 
320 &  1.01(-2)  [1.44] &  1.01(-2)  [1.44] F = 318&  1.05(-2)  [1.02] &   0.62 :  250 \\
\hline 
640 &  3.15(-3)  [1.68] &  3.15(-3)  [1.68] F = 567&  5.21(-3)  [1.01] &   0.31 :  125 \\
\hline 
1280 &  1.01(-3)  [1.64] &  1.14(-3)  [1.46] F = 941&  2.57(-3)  [1.01] &   0.16 :     62 \\
\hline 
\hline 
 \end{tabular}
 \caption{Anisotropic diffusion. Displayed are the errors for different timesteps along with estimated rates of convergence (in brackets). The numbers on the far right are the ``hyperbolic and parabolic CFL numbers''. The top box is for $m_1=m_2=99$, the middle for $m_1=m_2=199$ and the bottom for $m_1=m_2=799$. The F indicates how many times the Merge-adapt needed to add the BUG spaces. \label{tab:ani_errors}} 
 \end{center} 
 \end{table} 

\begin{figure}[htb!]
\begin{center}	
\includegraphics[width=0.49\textwidth,trim={0.0cm 0.0cm 0.0cm 0.0cm},clip]{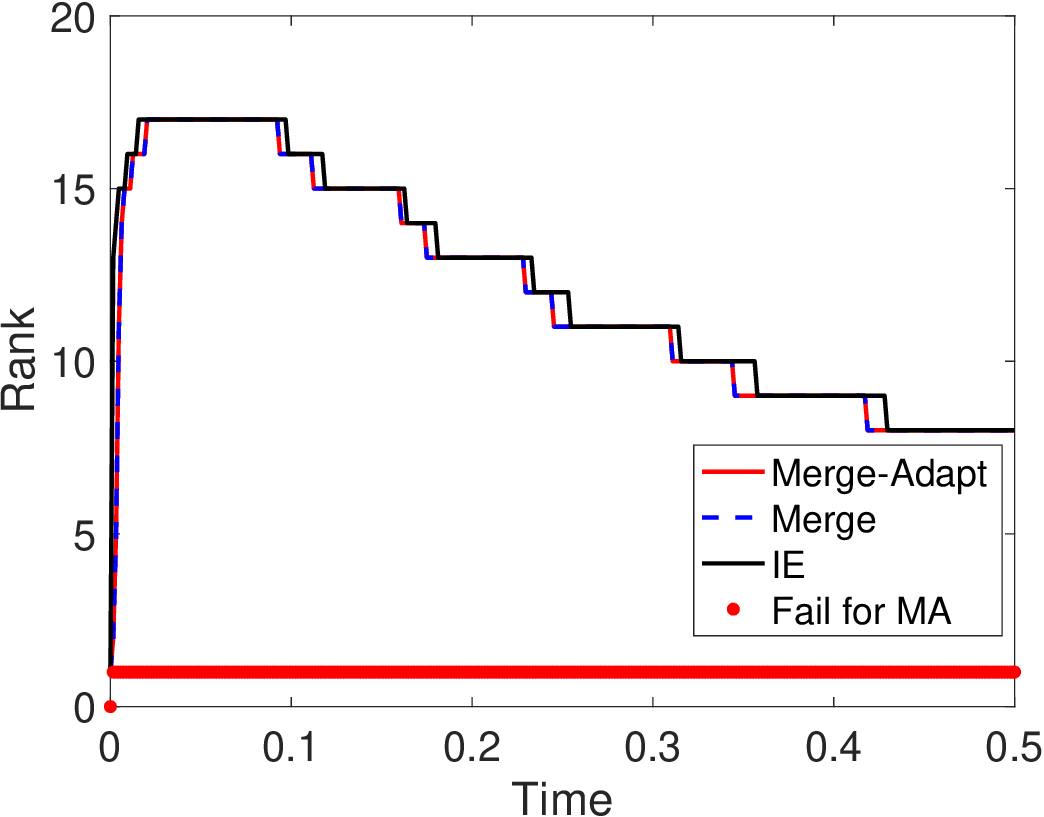}
\includegraphics[width=0.49\textwidth,trim={0.0cm 0.0cm 0.0cm 0.0cm},clip]{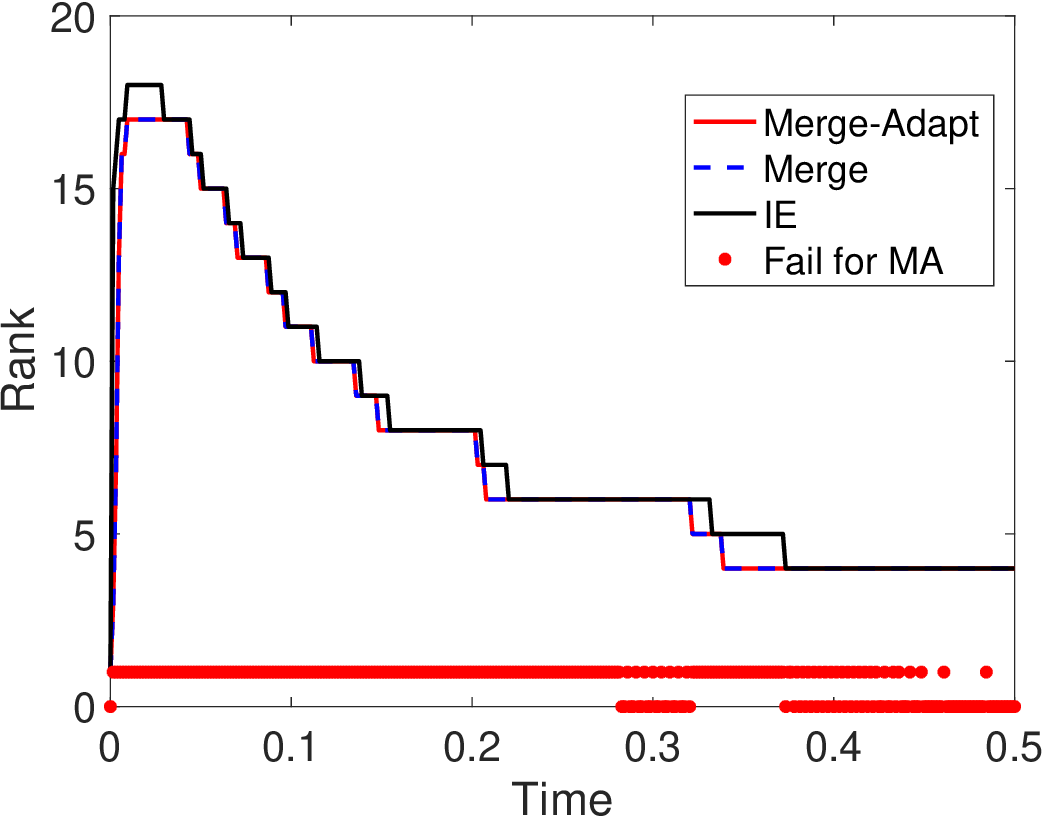}
\caption{Displayed is the {truncated} rank for the Merge, Merge-adapt and Implicit Euler method for the problem with  anisotropic diffusion initial data $ \sin( \pi x_1) \sin(\pi x_2),
$ (left) and  $\sin(2 \pi x_1) \sin(2\pi x_2)$ (right) and the solid body rotation problem (right). Here ``Fail for MA'' is the indicator where the BUG space is needed for the Merge-adapt method. For the Implicit Euler method we constantly use the threshold $\Delta t^2$ when computing the rank via the truncated SVD. All the computations are done with $m_1 = m_2 = 799$ and $n_{T} = 320.$ \label{fig:ani_rank}}
\end{center}
\end{figure}

In Table \ref{tab:ani_errors} we display the errors and rates of convergence. Again, the  methods are comparable although it should be noted that the rates of convergence for the classic implicit Euler method is more uniform than for the low rank methods. Curiously the error levels are smaller for the low rank method. It should also be noted that for this example the Merge-adapt method requires the addition of the BUG spaces almost every time step due to the stiffness of the problem.  We also track the rank of the solution as a function of time, the results can be found in  Figure \ref{fig:ani_rank}. The rank evolution is very similar for all methods.  

We also consider the more oscillatory initial data 
\[
\rho(0,x_1,x_2) = \sin( 2\pi x_1) \sin(2\pi x_2),
\] 
which we again evolve until time $t = 0.5$ using a base timestep of $\Delta t = 0.5 / n_T$, with $40, 80, 160, 320, 640, 1280$. Again, we carry out the computation on three different grids with $m_1=m_2 = 99, 399, 799$. 

This initial data decays more rapidly in time and as a consequence the residual norm becomes smaller and the Merge-adapt method does not have to add the BUG spaces as often. The errors and rates of convergence can be found in Table \ref{tab:ani_errors_sin2}, and the rank of the solution as a function of time, the results can be found in  Figure \ref{fig:ani_rot_rank}. It can be observed that for this example the classic implicit Euler method at first is a bit more accurate than the low rank methods but they catch up as the timestep decreases.  

\begin{table}[h!] 
 \begin{center} 
  \begin{tabular}{| p{0.8cm} | p{2.5cm} | p{4.2cm} | p{2.8cm}  |  p{2.5cm}  |}
 \hline
    $n_T $ & M & MA & IE & $\Delta t/h$ : $\Delta t/h^2$  \\
\hline 
\hline 
40 &  4.26(-1) &  1.00(0) F = 14&  7.58(-2) &   0.62 :     31 \\
\hline 
80 &  2.07(-1)  [1.03] &  1.00(0)  [0.00] F = 36&  3.80(-2)  [0.99] &   0.31 :     16 \\
\hline 
160 &  7.91(-2)  [1.39] &  1.95(-1)  [2.35] F = 60&  1.91(-2)  [0.99] &   0.16 :    7.8 \\
\hline 
320 &  2.35(-2)  [1.74] &  9.38(-2)  [1.05] F = 109&  9.67(-3)  [0.98] &  0.078 :    3.9 \\
\hline 
640 &  5.16(-3)  [2.19] &  9.77(-2)  [-0.0] F = 132&  4.97(-3)  [0.95] &  0.039 :      2 \\
\hline 
1280 &  1.01(-3)  [2.35] &  7.17(-2)  [0.44] F = 175&  2.68(-3)  [0.89] &   0.02 :   0.98 \\
\hline 
\hline 
40 &  4.87(-1) &  1.00(0) F = 14&  7.57(-2) &    1.2 :  128 \\
\hline 
80 &  2.62(-1)  [0.89] &  1.00(0)  [0.00] F = 36&  3.80(-2)  [0.99] &   0.62 :     62 \\
\hline 
160 &  1.13(-1)  [1.20] &  1.59(-1)  [2.64] F = 68&  1.90(-2)  [0.99] &   0.31 :     31 \\
\hline 
320 &  3.89(-2)  [1.54] &  6.80(-2)  [1.22] F = 134&  9.60(-3)  [0.98] &   0.16 :     16 \\
\hline 
640 &  1.06(-2)  [1.87] &  3.32(-2)  [1.03] F = 255&  4.87(-3)  [0.97] &  0.078 :    7.8 \\
\hline 
1280 &  2.02(-3)  [2.39] &  2.54(-2)  [0.38] F = 375&  2.52(-3)  [0.95] &  0.039 :    3.9 \\ 
\hline
\hline 
\hline 
40 &  5.85(-1) &  1.00(0) F = 13&  7.58(-2) &      5 :  2000 \\
\hline 
80 &  3.57(-1)  [0.71] &  1.00(0)  [0.00] F = 38&  3.80(-2)  [0.99] &    2.5 :  1000 \\
\hline 
160 &  1.81(-1)  [0.98] &  2.35(-1)  [2.08] F = 76&  1.91(-2)  [0.99] &    1.2 :  500 \\
\hline 
320 &  7.74(-2)  [1.22] &  8.88(-2)  [1.40] F = 175&  9.71(-3)  [0.97] &   0.62 :  250 \\
\hline 
640 &  2.89(-2)  [1.42] &  3.24(-2)  [1.45] F = 332&  5.04(-3)  [0.94] &   0.31 :  125 \\
\hline 
1280 &  9.18(-3)  [1.65] &  1.24(-2)  [1.38] F = 625&  2.78(-3)  [0.85] &   0.16 :     62.5 \\
\hline 
\end{tabular}
 \caption{Anisotropic diffusion for a higher frequency initial data. Displayed are the errors for different timesteps along with estimated rates of convergence (in brackets). The numbers on the far right are the ``hyperbolic and parabolic CFL numbers''. The top box is for $m_1=m_2=99$, the middle for $m_1=m_2=199$ and the bottom for $m_1=m_2=799$. The F indicates how many times the Merge-adapt needed to add the BUG spaces. \label{tab:ani_errors_sin2}} 
 \end{center} 
 \end{table}

\subsection{Timing results}
 \begin{table}[h!] 
 \begin{center} 
  \begin{tabular}{| p{1.7cm} | p{1.7cm} | p{1.4cm} |  p{1.7cm} | p{1.4cm} |  p{1.7cm} | p{1.4cm} | p{1.4cm} |}
 \hline
    $m_1 = m_2$ & time MA & time M & time MA  & time M  &time MA   & time M & time IE  \\
 \hline
    499 & 2 [0.1\%] & 18 & 6 [0.0\%] & 19 & 10 [45\%] & 28  & 1315\\
    \hline
    999 & 4 [0.1\%] & 68 & 12 [0.0\%] & 61 & 38 [48\%] & 260 & 153800\\
    \hline
 1999 & 11 [0.1\%] & 244 & 37 [0.0\%] & 129 &  535 [51\%] & 1645  & N/A\\
 \hline
 \end{tabular}
 \caption{CPU-times (in seconds) for the problems \ref{rot_example} \ref{ani_rot_example} and \ref{ani_heat_example} from lef to right. The numbers in brackets are the percentages for how often MA has a residual that necessitates the addition of the BUG space. \label{tab:CPU_TIMES}} 
 \end{center} 
 \end{table} 
In this experiment we repeat the experiments in Sections  \ref{rot_example} \ref{ani_rot_example} and \ref{ani_heat_example}. We use $n_T=1000$ and carry out computations on fine grids with $m_1=m_2 = 499, 999, 1999$.  We report the CPU times for the M and MA methods in Table \ref{tab:CPU_TIMES}. As can be seen the MA is generally faster than the M method. The more intensive use of the BUG space appears to happen when the ratio $\mu \Delta t/h^2$ is large. For the example in Section \ref{ani_heat_example} we have also added the time needed for the IE method for the two coarsest resolutions (the computation on the finest mesh could not fit on the computer used). For the IE timing we have only measured the time for the forward and backward substitutions in each timestep, not for the initial LU-factorization.

\section{Conclusions and future work}
\label{sec:conclusion}

This work prototypes a class of implicit  adaptive low rank time-stepping schemes. By simply merging the explicit step truncation and the BUG spaces, the modeling error is removed and the numerical schemes achieve robust convergence upon mesh refinement. An adaptive strategy is proposed for the prediction of row and column spaces, which is computationally advantageous for moderately stiff problems. The immediate future work is the generalization to  higher order in time. More importantly, the ideas can be applied to   tensor differential equations, which will be investigated in the future.

\section{Acknowledgement}
We would like to thank Steven R. White for pointing out the reference \cite{PhysRevB.102.094315}.

\section*{Declarations}
This research is partially supported by DOE Office of Advanced Scientific Computing Research under the Advanced Research in Quantum Computing program, subcontracted from award 2019-LLNL-SCW-1683, NSF DMS-2208164, NSF DMS-2011838, DOE grant DE-SC0023164 and Virginia Tech. The authors have no relevant financial or non-financial interests to disclose.

\section*{Data availability statement}
Data sets generated during the current study are available from the corresponding author on reasonable request. 

\bibliographystyle{abbrv}
\bibliography{lowrank}

\end{document}